\documentclass[12pt,a4paper]{amsart}
\usepackage[margin=1in]{geometry}
\usepackage{amsmath,amsthm,amsfonts,amssymb}
\usepackage{epigraph}
\usepackage{microtype}
\newcommand{\be}{\begin{equation}}
\newcommand{\ee}{\end{equation}}
\newcommand{\bes}{\begin{equation*}}
\newcommand{\ees}{\end{equation*}}
\newcommand{\bea}{\begin{eqnarray}}
\newcommand{\eea}{\end{eqnarray}}
\newcommand{\beas}{\begin{eqnarray}}
\newcommand{\eeas}{\end{eqnarray}}
\newcommand{\ben}{\begin{note}}
\newcommand{\een}{\end{note}}
\newcommand{\bexl}{\vskip0.1em\noindent\hrulefill\vskip1em\begin{ExerciseList}}
\newcommand{\eexl}{\end{ExerciseList}\hrulefill}

\newcommand{\bthm}{\begin{theorem}}
\newcommand{\ethm}{\end{theorem}}
\newcommand{\bpro}{\begin{prop}}
\newcommand{\epro}{\end{prop}}
\newcommand{\bcor}{\begin{corollary}}
\newcommand{\ecor}{\end{corollary}}
\newcommand{\bcon}{\begin{conjecture}}
\newcommand{\econ}{\end{conjecture}}
\newcommand{\bp}{\begin{proof}}
\newcommand{\ep}{\end{proof}}
\newcommand{\blem}{\begin{lemma}}
\newcommand{\elem}{\end{lemma}}
\newcommand{\bn}{\begin{note}}
\newcommand{\en}{\end{note}}
\newcommand{\benum}{\begin{enumerate}}
\newcommand{\eenum}{\end{enumerate}}
\newcommand{\bed}{\begin{defn}}
\newcommand{\eed}{\end{defn}}
\newcommand{\brem}{\begin{remark}}
\newcommand{\erem}{\end{remark}}

\newcommand{\btik}{\begin{tikzpicture}\begin{axis}[scale=0.5,axis y line=center, axis x line=middle]}
\newcommand{\etik}{\end{axis}\end{tikzpicture}}

\let\into=\hookrightarrow

\let\cong=\equiv

\newcommand{\upperRomannumeral}[1]{\uppercase\expandafter{\romannumeral#1}}

 \usepackage{graphicx}
\usepackage{xypic}
\usepackage{caption}
\usepackage{enumitem}
\usepackage{mathrsfs}
\usepackage{float}
\usepackage[colorlinks,pageanchor]{hyperref}
\usepackage{natbib}
\let\cite=\citep
\usepackage{zref-clever}
\let\Cref=\zcref

\newcommand{\Q}{{\mathbb Q}}
\newcommand{\R}{{\mathbb R}}
\newcommand{\Z}{{\mathbb Z}}
\newcommand{\C}{{\mathbb C}}

\renewcommand{\O}{{\mathcal O}}
\renewcommand{\H}{{\mathbb H}}
\let\implies=\Longrightarrow
\newcommand{\spec}{{\rm Spec}}
\newcommand{\End}{\rm End}

\newtheorem{theorem}[equation]{Theorem}\newtheorem{lemma}[equation]{Lemma}\newtheorem{corollary}[equation]{Corollary}\newtheorem{proposition}[equation]{Proposition}
\newtheorem{prop}[equation]{Proposition}
\theoremstyle{definition}
\newtheorem{conj}[equation]{Conjecture}
\newtheorem{conjecture}[equation]{Conjecture}
\newtheorem{remark}[equation]{Remark}
\newtheorem{example}[equation]{Example}
\newtheorem{definition}[equation]{Definition}
\newtheorem{defn}{Definition}
\newtheorem{question}[equation]{Question}
\numberwithin{equation}{subsection}

\zcsetup{cap} \AddToHook{env/theorem/begin}{\zcsetup{countertype={equation=theorem}}}
\AddToHook{env/corollary/begin}{\zcsetup{countertype={equation=corollary}}}
\AddToHook{env/lemma/begin}{\zcsetup{countertype={equation=lemma}}}
\AddToHook{env/proposition/begin}{\zcsetup{countertype={equation=proposition}}}
\AddToHook{env/prop/begin}{\zcsetup{countertype={equation=prop}}}
\AddToHook{env/theoremdef/begin}{\zcsetup{countertype={equation=theoremdef}}}
\AddToHook{env/defn/begin}{\zcsetup{countertype={equation=defn}}}
\AddToHook{env/definition/begin}{\zcsetup{countertype={equation=definition}}}
\AddToHook{env/remark/begin}{\zcsetup{countertype={equation=remark}}}
\AddToHook{env/rem/begin}{\zcsetup{countertype={equation=rem}}}
\AddToHook{env/conj/begin}{\zcsetup{countertype={equation=conj}}}
\AddToHook{env/conjecture/begin}{\zcsetup{countertype={equation=conjecture}}}
\AddToHook{env/question/begin}{\zcsetup{countertype={equation=question}}}
\AddToHook{env/example/begin}{\zcsetup{countertype={equation=example}}}

\zcRefTypeSetup{equation}{
	Name-sg = eq. ,
	name-sg = eq. ,
	Name-pl = eqns. ,
	name-pl = eqns. ,
}

\zcRefTypeSetup{theorem}{
	Name-sg = Theorem ,
	name-sg = theorem ,
	Name-pl = Theorems ,
	name-pl = theorems ,
}

\zcRefTypeSetup{theoremdef}{
	Name-sg = Theorem-Definition ,
	name-sg = theorem-definition ,
	Name-pl = Theorem-Definitions ,
	name-pl = theorem-definitions ,
}

\zcRefTypeSetup{corollary}{
	Name-sg = Corollary ,
	name-sg = corollary ,
	Name-pl = Corollaries ,
	name-pl = corollaries ,
}

\zcRefTypeSetup{proposition}{
	Name-sg = Proposition ,
	name-sg = proposition ,
	Name-pl = Propositions ,
	name-pl = propositions ,
}

\zcRefTypeSetup{prop}{
	Name-sg = Proposition ,
	name-sg = proposition ,
	Name-pl = Propositions ,
	name-pl = propositions ,
}

\zcRefTypeSetup{lemma}{
	Name-sg = Lemma ,
	name-sg = lemma ,
	Name-pl = Lemmas ,
	name-pl = lemmas ,
}

\zcRefTypeSetup{remark}{
	Name-sg = Remark ,
	name-sg = remark ,
	Name-pl = Remarks ,
	name-pl = remarks ,
}

\zcRefTypeSetup{question}{
	Name-sg = Question ,
	name-sg = question ,
	Name-pl = Questions ,
	name-pl = questions ,
}

\zcRefTypeSetup{example}{
	Name-sg = Example ,
	name-sg = example ,
	Name-pl = Examples ,
	name-pl = examples ,
}

\zcRefTypeSetup{defn}{
	Name-sg = Definition ,
	name-sg = definition ,
	Name-pl = Definitions ,
	name-pl = definitions ,
}

\zcRefTypeSetup{conj}{
	Name-sg = Conjecture ,
	name-sg = conjecture ,
	Name-pl = Conjectures ,
	name-pl = conjectures ,
}

\zcRefTypeSetup{conjecture}{
	Name-sg = Conjecture ,
	name-sg = conjecture ,
	Name-pl = Conjectures ,
	name-pl = conjectures ,
}

\begin{document}
\numberwithin{equation}{subsection}
\newcommand{\ra}{\rightarrow}

\title{On primes of ordinary and Hodge-Witt reduction}
\author{Kirti Joshi}

\address{Department of
Mathematics, University of Arizona, 617 N Santa Rita, P O Box 210089,
Tucson, AZ 85721, USA.}


\let\tensor=\otimes
\newcommand{\Hom}{{{\rm Hom}}}
\newcommand{\Spec}{{{\rm Spec}}}
\let\tensor=\otimes
\let\isom=\equiv \renewcommand{\P}{{\mathbb P}}
\let\liminv=\varprojlim
\renewcommand{\wp}{\mathfrak{p}}
\newcommand{\wq}{\mathfrak{q}}
\newcommand{\dor}[2]{\delta_{#1,#2}^{\rm ord}}
\newcommand{\dhw}[2]{\delta_{#1,#2}^{\rm hw}}
\newcommand{\dnhw}[2]{\delta_{#1,#2}^{\rm non-hw}}
\newcommand{\dora}[1]{\delta_{#1}^{\rm ord}}
\newcommand{\dhwa}[1]{\delta_{#1}^{\rm hw}}
\newcommand{\dnhwa}[1]{\delta_{#1}^{\rm non-hw}}
\let\cong=\equiv
\newcommand{\gal}{{\rm Gal}}
\newcommand{\frob}{{\rm Frob}}

\newcommand{\br}{\begin{remark}}
\newcommand{\er}{\end{remark}}
\newcommand{\benumlab}{\begin{enumerate}[label={{\bf(\arabic{*})}}]}
\newcommand{\mustata}{Musta\c t\v a}
\begin{abstract}
Jean-Pierre Serre has conjectured \Cref{serre}, in the context of abelian varieties, that there are infinitely primes of good ordinary reduction for a smooth, projective variety over a number field. We prove this conjecture for K3 surfaces \Cref{th:positive-density} (this is unpublished joint result with C.~S.~Rajan which was also independently established by Fedor Bogomolov and Yuri Zarhin by a different method). Any prime of ordinary reduction is also a prime of Hodge-Witt reduction but not conversely. \Cref{con:positive-density-hodge-witt} (of Joshi-Rajan) asserts the existence of infinitely many primes of Hodge-Witt reduction.  The two conjectures are related but not equivalent (\Cref{th:equivalence}). We prove the latter conjecture for abelian threefolds \Cref{th:positive-density-hodge-witt-abelian-threefold} (joint with C.~S. Rajan),  and smooth Fano threefolds (\Cref{th:fano-threefold}) and in \Cref{th:CM-ab-general} for abelian varieties with complex multiplication. We show that the set of primes of ordinary and Hodge-Witt reduction can have different densities (\Cref{thm:cm-hodge-witt}, \Cref{elliptic-curve-example}).    \Cref{th:ord-wonderful-compact},  \Cref{th:ord-config-spaces} establish the existence of ordinary reductions for certain wonderful compactifications and a large class of configuration spaces. \Cref{th:jr-ms} deals with the relationship between \Cref{con:positive-density-hodge-witt}, \Cref{serre} and the \mustata-Srinivas conjectures and \Cref{cor:mustata-srinivas-cm} establishes this conjecture in some cases. \Cref{se:non-hw} deals with existence of primes of non Hodge-Witt reductions and establishes this in a number of cases (\Cref{th:fermat-non-hw1}), and \Cref{th:frequency-of-non-hodge-witt-cm} asserts that for Fermat hypersurfaces of dimension $\geq 3$ and degrees $\geq 211$, at least $98\%$ of the primes are of non Hodge-Witt (and hence non-ordinary) reduction and in the degree $\to\infty$ limit, almost all primes are of non Hodge-Witt (and hence non-ordinary) reduction. \Cref{se:examples} provides a number of examples illustrating densities of ordinary, Hodge-Witt and non Hodge-Witt reductions.
\end{abstract}

\maketitle
\tableofcontents
\epigraph{\includegraphics[scale=.7]{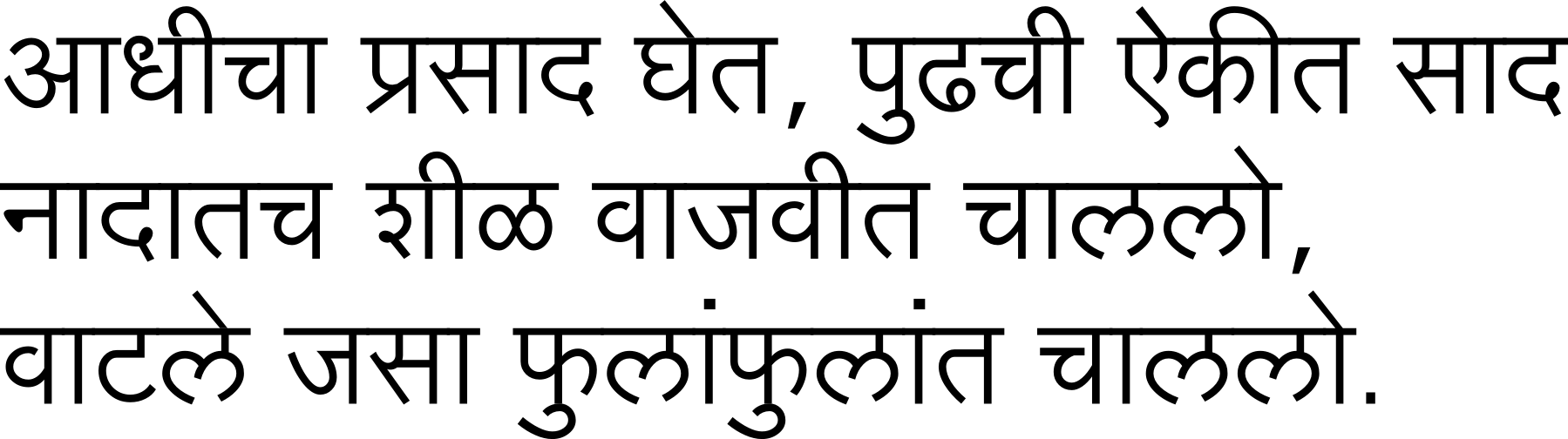}}{\includegraphics[scale=.7]{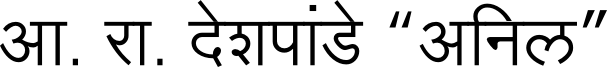} }
\renewcommand{\leftmark}{On Ordinary and Hodge-Witt reduction}
\renewcommand{\rightmark}{Kirti Joshi}

\section{Introduction}\label{intro}
\subsection{Motivating questions}
 Let $k$ be a perfect field of characteristic $p>0$. An abelian
variety $A$ over $k$ is said to be \emph{ordinary} if the $p$-rank of $A$ is
the maximum possible, namely equal to the dimension of $A$. The notion
of ordinarity was extended in \cite{Mazur1972}, to a smooth
projective variety $X$ over $k$, using notions from crystalline
cohomology. A more general definition was given by in
\cite[Definition 7.2]{bloch1986}, \cite[IV, D\'efinition 4.12]{Illusie1983a} using coherent
cohomology.   
Ordinary varieties tend to have special
properties and play a key role in the study of varieties over fields of
characteristic $p>0$. For example,  the existence of canonical Serre-Tate lifting
for ordinary abelian varieties to characteristic $0$; the comparison
theorems between crystalline cohomology and $p$-adic \'etale
cohomology were first established for such varieties \cite{bloch1986}; ordinarity is an open condition on the base \cite{illusie90a} and so on.  

One of the  motivating questions for this paper is \Cref{serre}, which is commonly attributed to Jean-Pierre Serre (see \cite{ogus82}), and as far as I am aware, it was  originally formulated for
abelian varieties, and it asserts the following: let $K$ be a number field, and let $X$ be a smooth, projective variety defined over $K$. Then there exists  a positive density of primes of
$K$, at which $X$ has ordinary reduction.   

 The class of Hodge-Witt varieties (\Cref{def:hw}) was introduced in \cite[IV, 4.6]{Illusie1983a}, and are characterized as varieties whose  de Rham-Witt cohomology groups $H^i(X,
 W\Omega^j_X)$ are finitely generated over ring of Witt vectors $W$ of the base field, and this class  strictly contains the class of ordinary varieties  (for the equivalent characterizations of Hodge-Witt abelian varieties see \ref{thm:cm-hodge-witt-3}). While it has been understood that ordinary varieties have different geometrical properties from general varieties, it is not commonly understood that Hodge-Witt varieties also have distinct geometrical properties from varieties which are non-Hodge-Witt.  An example of this behavior is implicit in the work of Chai-Conrad-Oort (see \cite{conrad-book}), where the term `Hodge-Witt' is not mentioned, but some of the lifting/non-lifting phenomena studied therein are dependent on whether the abelian variety in question is Hodge-Witt or non-Hodge-Witt. Hence, we provide a translation between the point of view of \cite{conrad-book} and our point of view in Theorem~\ref{thm:trans-conrad-book}. 

In \Cref{se:k3-ord}, we establish the existence of ordinary reductions given by \Cref{serre} for several classes of varieties. Ordinarity of reductions of certain wonderful compactifications   and configuration spaces arising from varieties considered in \Cref{se:k3-ord} are established in \Cref{th:ord-wonderful-compact}, \Cref{th:ord-config-spaces}, and their corollaries respectively.

Generalizing Serre's ordinarity \Cref{serre}, we conjecture in \Cref{con:positive-density-hodge-witt}, that a smooth, projective variety $X/K$ over a number field
$K$, there exist infinitely many primes of Hodge-Witt reduction.  This conjecture does not appear to have been studied before, and we investigate it in \Cref{se:hodge-witt-reduction}. \Cref{serre} of Serre and our \Cref{con:positive-density-hodge-witt} are not equivalent as there exists a smooth, projective varieties over $\Q$ such that the sets of primes provided by Conjecture~\ref{serre} and \Cref{con:positive-density-hodge-witt} are of different densities (see Example~\ref{elliptic-curve-example}). But the two Conjectures are almost equivalent and the precise relationship between the two conjectures is Theorem~\ref{th:equivalence}.  \Cref{elliptic-curve-example}, \Cref{elliptic-curve-example2}  show that the two conjectures provides primes of different densities even though they are related. 

Next we give some evidence for \Cref{con:positive-density-hodge-witt} for smooth, projective varieties $X/K$ by proving it in the following cases:
\benumlab 
\item $X$ is a surface with $p_g=0$ (\Cref{pgtheorem}),
\item $X$ is a Fano surface (\Cref{fanotheorem}) or $X/K$ is  Fano threefold over a number field (\Cref{th:fano-threefold}),
\item  $X$ is an  abelian threefold over number
fields \Cref{th:positive-density-hodge-witt-abelian-threefold} (this is joint work with C.~S.~Rajan),
\item $X/K$ is an  abelian variety with Complex Multiplication \Cref{th:CM-ab-general} (and \Cref{thm:cm-hodge-witt-2}). This requires proving a criterion (\Cref{thm:cm-hodge-witt}) for Hodge-Witt reductions of a simple abelian variety with CM.
\item Fermat hypersurfaces in \Cref{th:fermat-non-hw1}{\bf(1)}.
\eenum

In \cite{mustata12a,mustata12b}, the existence of ordinary reductions is related to a relationship between test ideals and multiplier ideals (of subschemes) and this  has also led to some interest in ordinary varieties. In \Cref{th:jr-ms} we spell out the relationships between \Cref{serre} (Serre), \Cref{con:positive-density-hodge-witt} (Joshi-Rajan) and \cite[Conjectures 1.1 and 1.2]{mustata12a}.  Corollary~\ref{cor:mustata-srinivas-cm} asserts that the conjecture of \mustata-Srinivas (see \cite[Conjecture~1.2]{mustata12a,mustata12b}) is true for a number of configuration spaces arising from some varieties whose ordinarity is established in this paper (including abelian varieties with complex multiplication,  and Fermat varieties of general type).  As far as I am aware, this is the only general result known in the direction of \cite[Conjecture~1.2]{mustata12a,mustata12b}.

Now let discuss the main results of \Cref{se:k3-ord}.
Some of the results of \Cref{se:k3-ord} were included in my  
preprint \cite[Section~6 and 7]{joshi00a-arxiv-version} with C.~S.~Rajan. However, when that preprint was
accepted for publication in the International Math. Res. Notices
(see \cite{joshi03}), the referee suggested the removal of these sections as Theorem~\ref{th:positive-density}(1)  might be well-known to experts and at any rate the contents
of sections 6,7 of that preprint be published independently because the
results contained in these two sections were independent of the main
results of that paper (namely construction of examples of Frobenius
split non-ordinary varieties). Hence, the results of this \Cref{se:k3-ord} do not appear in \cite{joshi03} and remained unpublished. Subsequently  (Theorem~\ref{th:positive-density}(1)) was proved by Fedor~Bogomolov and Yuri~Zarhin (see \cite{bogomolov09}) by a different method. Both the approaches (ours and that of \cite{bogomolov09}) are important. For a non-trivial geometric application of Theorem~\ref{th:positive-density}(1) see the S\'eminaire Bourbaki lecture \cite{benoist15} on the construction of rational curves on K3 surfaces by \cite{bogomolov11}, \cite{liedtke12}, and others. 
Our methods are different from those of \cite{bogomolov09}, \cite{tankeev95} and proves, in addition to establishing the existence of primes of ordinary reduction for K3 surfaces, a crucial step (see Theorem~\ref{th:positive-density}(2)) in the proof of  existence of primes of Hodge-Witt reduction for abelian threefolds (\Cref{th:positive-density-hodge-witt-abelian-threefold})--thus verifying  \Cref{con:positive-density-hodge-witt} for all abelian threefolds. This is certainly a new result of independent interest. The proofs of these two results are so closely linked that it would seem artificial to separate them. This is our rationale for publishing these hitherto unpublished results with C.~S.~Rajan.  

For purposes of citation, we provide joint attributions to results obtained with C.~S.~Rajan; all the unattributed or uncited results are mine.

In \Cref{se:k3-ord}, we have also added to the results of the aforementioned sections of \cite{joshi00a-arxiv-version},  Theorem~\ref{th:equivalence} which clarifies the precise relationship between Conjecture~\ref{serre} and \Cref{con:positive-density-hodge-witt}

In \Cref{se:non-hw}, we study the subtler question of the presence of infinite torsion in the de Rham-Witt cohomology of varieties and discuss some questions and conjectures.  In particular,  \Cref{con:non-hodge-witt}{\bf(1)} asserts that for any smooth, projective variety $X$ of dimension $n=\dim(X)$ over a number field $K$,  there exist infinitely many primes of $K$ at which $X$ has good, non Hodge-Witt reduction. Conjecture~\ref{con:non-hodge-witt}{\bf(2)} asserts that if  $H^n(X,\O_X)\neq0$, then the de Rham-Witt differential $H^n(X_\wp,W(\O_{X_\wp}))\to H^n(X_\wp,W\Omega^1_{X_\wp})$ (of the reduction $X_\wp$) is not zero (and hence $H^n(X_\wp,W(\O_{X_\wp}))$ has infinite torsion) for a positive density of primes $\wp$ of $K$ (note that Conjecture~\ref{con:non-hodge-witt}{\bf(2)}$\implies$Conjecture~\ref{con:non-hodge-witt}{\bf(1)}). As we note in \Cref{con:non-hodge-witt-implies-elkies-thm}{\bf(1)}, Conjecture~\ref{con:non-hodge-witt}{\bf(1)} includes, as a special case, the conjecture that there exists infinitely many primes of good, supersingular reduction for any elliptic curve over $K$ (if $K\into \R$, this is the well-known theorem of \cite{elkies87,elkies89}), and  also  implies (see \Cref{con:non-hodge-witt-implies-elkies-thm}{\bf(2)})  that any pair of  elliptic curves over \textit{any} number field $K$ have infinitely many common primes of supersingular reduction, and in \Cref{pr:ab-var-non-hw}  any abelian variety of dimension $g$ over a number field has infinitely many primes where the reduction $p$-rank $\leq  g-2$. In \Cref{th:fermat-non-hw1} we show \Cref{con:non-hodge-witt} is true for that Fermat hypersurfaces of dimension two or three and all simple abelian varieties with complex multiplication

Perhaps the most counter-intuitive assertion which we prove in the context of \Cref{con:non-hodge-witt} is \Cref{th:frequency-of-non-hodge-witt-cm}, where we show that for any smooth, Fermat hypersurface $F_{n,m}\subset \P^{n+1}$ (defined over $\Q$) of dimension $n\geq 3$ and degree $m\geq 211$, the density of primes of non-Hodge-Witt reduction is at least $\geq 98\%$ (and in the limit $m\to\infty$, the density of primes of non Hodge-Witt reduction is equal to one). In particular, \Cref{con:non-hodge-witt} is true for $F_{n,m}$ (\Cref{th:fermat-non-hw1}). Thus,  even for the familiar and extensively studied class of varieties, there is a surprising prevalence of primes of non Hodge-Witt reduction i.e. most primes are primes of non-Hodge-Witt reduction (and hence also of non-ordinary reduction). 

In \Cref{se:examples} we provide several explicit examples and density calculations/estimates  for all the phenomena studied in this paper: Fermat hypersurfaces (\Cref{ss:fermat}, \Cref{ss:prevalence}), \Cref{ss:abvar-gen-typ-surf} for abelian varieties (\Cref{elliptic-curve-example}, \Cref{elliptic-curve-example2}, \Cref{ex:hodge-witt-cm-galois-example}, \Cref{ex:hodge-witt-cm-example}) and some surfaces of general type (\Cref{surf-gen-type}) and abelian threefolds (\Cref{ex:ab-thrfld}), abelian fourfolds (\Cref{ex:ab-fourfold}, \Cref{ex:ab-fourfold2}, \Cref{ex:ab-fourfold3}) and  high-dimensional varieties $X$ of general type (\Cref{ex:gen-type-negligible-den}) with $0<\delta^{\rm ord}_X<\delta^{\rm hw}_X<\varepsilon$, while $\delta^{hw}_X/\delta^{ord}_X> c>1$ (for any given real numbers $0<\varepsilon<1$ and $c>1$).

\subsection{Acknowledgments} \Cref{se:k3-ord} was  written while I was visiting, the Tata Institute of Fundamental Research (a very long time ago) and support and hospitality of the institute is gratefully acknowledged. I thank Vikram Mehta for conversations during that period. I thank C.~S.~Rajan for conversations and for allowing us to include our previously unpublished joint results here. I would also like to thank Karl Schwede and Kevin Tucker, organizers of the 2011 AIMS  workshop on ``Relating test ideals and multiplier ideals'' for giving me an opportunity to speak on results of \cite{joshi00a-arxiv-version}. We decided to revise the unpublished sections of \cite{joshi00a-arxiv-version} because of this workshop and the support and hospitality of AIMS is gratefully acknowledged. 

It is a pleasure to thank the Referee for spotting errors, inadequacies of the previous version(s) and suggesting  corrections, a number of improvements and \Cref{ex:ab-fourfold}.

\section{Preliminaries}\label{preliminaries}
\subsection{Ordinary varieties}\label{ordinary-subsection}
Let $X$ be a smooth projective variety over a perfect field $k$ of
positive characteristic.  Following Bloch-Kato \cite{bloch1986} and
Illusie-Raynaud \cite{Illusie1983a},  we say that $X$ is ordinary  if
\be\label{eq:ord} H^i(X,B^j_X)=0\text{ for all }i\geq 0,\, j>0,\ee where  \[ B^j_X= {\rm
image}\left( d:\Omega^{j-1}_X\to \Omega^j_X\right).\] If $X$ is an
abelian variety, then it is known that this definition  coincides with
the usual definition \cite[Proposition 7.3 and Example 7.4]{bloch1986}.    By
\cite[Proposition 1.2]{illusie90a}, ordinarity is an  open condition in the
following sense: if $X\to S$ is a smooth, proper family of varieties
parameterized by $S$, then the set of points $s$ in  $S$, such that the
fiber $X_s$ is ordinary is a Zariski open subset of $S$.

Following \cite{mehta1985}, one says that  $X$ is \textit{Frobenius split} (or \textit{$F$-split}) if the exact sequence of $\O_X$-modules
\begin{equation}\label{eq:f-seq}
0\to \O_X\to F_*(\O_X) \to B^1_X\to 0.
\end{equation} splits. An abelian variety $X$ is Frobenius split i.e. $F$-split if and only if $X$ is an ordinary abelian variety. For the relationship between ordinarity and Frobenius splitting see \cite{joshi03}.

\subsection{Hodge-Witt varieties}
   The standard reference for de Rham-Witt cohomology is
\cite{illusie79b}. Throughout this section,  the following notations
will be in force. Let $k$ be an algebraically closed field of
characteristic $p>0$, and $X$ a smooth, projective variety over $k$.
Let $W=W(k)$ be the ring of Witt vectors of $k$. Let $K=W[1/p]$ be
the quotient field of $W$. Note that as $k$ is perfect, $W$ is a
Noetherian local ring with a discrete valuation and with residue
field $k$. For any $n\geq 1$, let $W_n=W(k)/p^n$. $W$ comes equipped
with a lift $\sigma:W\to W$, of the Frobenius morphism of $k$, which
will be called the Frobenius of $W$. We define a non-commutative
ring $R^0=W_\sigma[V,F]$, where $F,V$ are two indeterminate subject
to the relations $FV=VF=p$ and $Fa=\sigma(a)F$ and $aV=V\sigma(a)$.
The ring $R^0$ is called the Dieudonne ring of $k$. The notation $R^0$ for this ring is
borrowed from \cite{Illusie1983a}.

     Let $\{W_n\Omega^*_X\}_{n\geq 1}$ be the de Rham-Witt pro-complex
constructed in \cite[Chapitre I]{illusie79b}. It is standard that for each
$n\geq 1,i,j\geq 0$, $H^i(X,W_n\Omega^j_X)$ are of finite type over
$W_n$. We define \cite[II, Proposition 2.1]{illusie79b}
\begin{equation}\label{eq:drcohom}
    H^i(X,W\Omega^j_X)=\liminv_{n}H^i(X,W_n\Omega^j_X).
 \end{equation} 
The groups defined by \eqref{eq:drcohom} are $W$-modules, and modulo  torsion, these groups are finite type $W$-modules (see \cite[II, Th\'eor\`eme 2.13]{illusie79b}). The cohomology groups \eqref{eq:drcohom} are called \textit{Hodge-Witt cohomology groups  of
$X$}.

\begin{definition}\label{def:hw} We say that $X$ is Hodge-Witt, if for $i,j\geq 0$,
the Hodge-Witt cohomology
 groups $ H^i(X,W\Omega^j_X)$ are finite type $W$-modules.
\end{definition}
\newcommand{\mydot}{{\scriptsize\bullet}}
The properties of the de Rham-Witt pro-complex $\left\{W_n\Omega_X^*\right\}_{n\geq 1}$ are reflected  in
these cohomology modules and in particular, we note that for each
$i,j$, the Hodge-Witt groups $H^i(X,W\Omega_X^j)$ are left modules
over $R^0$. The complex $W\Omega^*_X$,  defined naturally
using the de Rham-Witt pro-complex, computes the crystalline
cohomology of $X$ \cite[II, Th\'eor\`eme 1.4]{illusie79b} and in particular, there is a spectral sequence
\begin{equation}\label{eq:slope-ss}
   E_1^{i,j}=H^i(X, W\Omega^j_X)\Rightarrow H^*_{\rm cris}(X/W)
\end{equation}
which is called the \textit{slope spectral sequence of $X$} \cite[II, 3.1]{illusie79b}.
The slope spectral sequence \eqref{eq:slope-ss} induces a filtration on the crystalline
cohomology $H^*_{cris}(X/W)$ of $X$ which is called the \textit{slope filtration} (\cite[II, 3.1.2]{illusie79b}). It is standard  that the slope spectral sequence degenerates at
$E_1$ modulo torsion (i.e. the differentials are zero on tensoring
with $K$) (see \cite[II, Theorem~3.2]{illusie79b}) and at $E_2$ up to finite length (i.e. all the
differential have images which are of finite length over W) (see 
\cite[Chapter II, Corollary 3.2]{Illusie1983a}).

For a general, smooth proper variety $X$, the de Rham-Witt cohomology groups are
not of finite type over $W$, and the structure of these groups
reflects the arithmetical properties of $X$ \cite[II, 7.2]{illusie79b}. For instance, in
\cite{bloch1986}, \cite[Chapitre IV, Th\'eor\`eme 4.13]{Illusie1983a} it is shown that for ordinary
varieties $H^i(X,W\Omega^j)$ are of finite type over $W$.

Let us also recall the following fact which holds in dimension one and
which will be used in the rest of the paper:

\bpro\label{pr:curves-are-hodge-witt} 
Let $X/k$ be a geometrically connected, smooth, proper curve over a perfect field of characteristic $p>0$. Then 
\benumlab
\item $X$ is ordinary if and only if its Jacobian is ordinary.
\item $X$ is (always) Hodge-Witt i.e. a smooth, projective curve over a perfect field of characteristic $p>0$ is  Hodge-Witt.
\eenum
\epro
\bp
The first assertion is certainly well-known to experts, but let me give a proof. Let $A$ be the dual of $J=Pic^0(X)$ i.e. $A$ is the Albanese variety of $X$. To prove {\bf(1)}, it is enough to prove that Frobenius morphism of $X$ on $H^1(X,\O_X)\to H^1(X,\O_X)$ is an isomorphism  if and only the Frobenius morphism of $A$ induces an isomorphism $H^1(A,\O_A)\to H^1(A,\O_A)$.  By \cite[Proposition 7.3 and Example 7.4]{bloch1986}, the  last assertion for $A$ is equivalent to the assertion that \eqref{eq:ord} holds for $A$. Now we have a natural functorial isomorphism $H^1(X,\O_X)\to H^1(A,\O_A)$. This follows from \cite[II, Proposition 5.16]{illusie79b}.

Now if $H^1(X,\O_X)\to H^1(X,\O_X)$ is an isomorphism, then cohomology applied to \eqref{eq:f-seq} shows that $H^1(X,B_X^1)=0$ and the fact that Frobenius of $X$ is a finite morphism shows $H^0(X,F_*(\O_X))=H^0(X,\O_X)=k$ and thus $H^0(X,B_X^1)=0$. Thus, $X$ is ordinary if and only $H^1(X,\O_X) \to H^1(X,\O_X)$ is an isomorphism if and only if \eqref{eq:ord} holds if and only if $A$ is ordinary if and only if $J$ is ordinary.

The assertion {\bf(2)} is less well-know but is immediate from \cite[II, Corollaire 2.17 et Collaire 2.18]{illusie79b}.
\ep

\section{Primes of ordinary reductions}\label{se:k3-ord} 
In this section, we discuss \Cref{serre}, due to Serre, on the existence of primes of good ordinary reductions; the main theorem \Cref{th:positive-density} (joint with Rajan) establishes the existence of primes of ordinary reductions for K3 surfaces over number fields. Other ordinarity results we establish in this paper are as follows. In \Cref{th:CM-ab-general}, \Cref{thm:cm-hodge-witt} we will establish the existence of ordinary reductions for abelian varieties with complex multiplication. In \Cref{th:ord-wonderful-compact} we establish the existence of ordinary reductions for a large class of wonderful compactifications, and this allows us to assert the existence of ordinary reductions for a large class of  configuration spaces in \Cref{th:ord-config-spaces}.

\subsection{Models}\label{ss:models}
Let $X$ be a smooth projective variety over a number field $K$. By an \emph{$\O_K$-model for $X/K$} or simply a \emph{model for $X/K$} we mean a  scheme $f:\mathcal{X}\to\Spec(\O_K)$  such that
\benumlab
\item  structure morphism $f$ is projective, and
\item  the generic fiber $\mathcal{X}_K=\mathcal{X}\times_{\Spec(\O_K)}\Spec(K)\simeq X$, and
\item there exists some non-empty Zariski open subset $\Spec(\O_K)-V(I) \subseteq\Spec(\O_K)$ (here $0\neq I\neq\O_K$ is an ideal of $\O_K$) over which $f$ is flat with smooth fibers.
\eenum 
Such models evidently exist (see \cite[Chapter 1]{bosch-neron-models}).  In what follows, if $\wp\in\Spec(\O_K)-V(I)$ is a prime ideal, then the fiber of this model over $\wp$ will be denoted by $X_{\wp}$. In what follows, one may further shrink the Zariski open subset $\Spec(\O_K)-V(I)$ given above by {\bf(3)} so that certain additional properties hold. Up to such shrinking, all the properties  and results considered here are independent of the choice of such a model.

\subsection{Serre's conjecture}
The following more general question, which is one of the
motivating questions for this paper, is the following conjecture which
is well-known and was raised initially for abelian varieties by Jean-Pierre Serre \cite{ogus82}:
\begin{conj}\label{serre}
Let $X/K$ be a smooth projective variety over a number field $K$. Then
there is a positive density of primes $v$ of $K$ for which $X$ has
good ordinary reduction at $v$.
\end{conj}
Note that this question obviously requires a model as discussed in \Cref{ss:models}. From now on, we will assume that one has a model of the type considered in \Cref{ss:models} at our disposal.

\subsection{Primes of ordinary reduction for K3 surfaces}
Our aim in this section is
to prove the following theorem which was proved in \cite{joshi00a-arxiv-version}.
\begin{theorem}[Joshi-Rajan]\label{th:positive-density}
	Let $X$ be a $K3$ surface or an abelian variety of dimension at
	least two defined over a number field $K$. Then there is a finite
	extension $L/K$ of number fields, such that
	\begin{enumerate}
		\item if $X$ is a $K3$-surface, then $X\times_K L$ has
		ordinary reduction at a set of primes of density one in $L$.
		\item if $X$ is an abelian variety, then there is set of
		primes $O$ of density one in $L$, such that the reduction of
		$X\times_{K}L$ at a prime $p\in O$ has $p$-rank at least two.
	\end{enumerate}
\end{theorem}

\brem\  
\benumlab
\item Note that the set of primes given by \Cref{th:positive-density} is of density one in the set of primes of $L$, but the set of primes $K$ lying below these primes of $L$ is of positive density in the set of primes of $K$. This should not cause any confusion--common understanding of \Cref{serre} usually refers to density of the set of primes of $K$.
\item The proof closely follows the method of \cite[page 372]{ogus82}.
\item We note here that a proof of the result for a class of K3
surfaces was also given in \cite{tankeev95} under somewhat restrictive hypothesis. The question of primes of ordinary
reduction for abelian varieties has also been treated  by
\cite{noot95}, \cite{pink98}; and more \cite{vasiu00} has studied the question for a
wider class of varieties. The approach adopted by these authors is
through the study of Mumford-Tate groups. 
\item In \cite{bogomolov09}, \Cref{th:positive-density}{\bf(1)} (K3 surface case) was also independently established by a different method. 
\eenum
\erem

Until the end of this section, the following conventions will be in force unless stated otherwise: let $K$ be a number field, and let $X$ denote either an abelian
variety or a $K3$ surface defined over $K$.  Let $\O_K$ be the ring of integers of $K$; for a prime
$\wp$ of $\O_K$ lying above a rational prime $p$, $K_\wp$ be the completion of $K$ at $\wp$. For notational simplicity, let $\O_{\wp}=\O_{K_\wp}$ and let $k_{\wp}$ be the residue
field of $\wp$ and let  $q_\wp=p^{c_\wp}$ be the cardinality of $k_\wp$. Let $W(k_\wp)$ denote the ring of Witt vectors of $k_\wp$. Fix a model $\mathcal{X}\to\Spec(\O_K)$ as in \Cref{ss:models}. Assume that $\wp$ is a place
of good reduction for $X$ as above and write $X_\wp$ for the reduction
of $X$ at $\wp$.  We recall here the following facts:

\setcounter{subsubsection}{\value{equation}\refstepcounter{equation}}

\subsubsection{Trace of Frobenius}
   The Frobenius endomorphism $F_\wp$ is an  endomorphism
 of the $l$-adic cohomology groups $H^i_\ell:= H^i_{\acute{e}t}(X\otimes
 \overline{K}, \Q_\ell)$ for a prime $\ell\neq p$.  The $\ell$-adic
 characteristic polynomial $P_{i, \wp}(t)= {\rm det}(1-tF_\wp\mid
 H^i_{\acute{e}t}(X\otimes \overline{K}, \Q_\ell))$ is an integral polynomial and
 is independent of $\ell$. Let
\[ a_\wp={\rm Tr}(F_\wp\mid H^2_{\acute{e}t}(X\otimes \overline{K}, \Q_\ell))\]
denote the trace of the $\ell$-adic Frobenius acting on the second
\'etale cohomology group.  Then $a_\wp$ is a rational integer (see
\cite[Page 370]{ogus82}).

\setcounter{subsubsection}{\value{equation}\refstepcounter{equation}}

\subsubsection{Deligne Weil estimate} (Deligne-Weil estimates) \cite{deligne74a}:
It follows from Weil estimates proved by Weil for abelian varieties and
by Deligne in general  that
\[ |a_\wp|\leq dp\]
where $d=\dim H^2_\ell$ is a constant independent of the place $\wp$.

\setcounter{subsubsection}{\value{equation}\refstepcounter{equation}}

\subsubsection{Katz-Messing theorem}
 Let $\phi_\wp$ denote the
crystalline Frobenius on $H^i_{cris}(X/W(k_\wp))$. Then $\phi_\wp^{c_\wp}$ is
linear over $W(k_\wp)$ (here $q_\wp=p^{c_\wp}$ is the cardinality of the residue field $k_\wp$), and  the characteristic polynomials of
the (linearized) crystalline Frobenius and
the $\ell$-adic Frobenius are equal:
$$P_{i,\wp}(t)={\rm det}(1-t\phi_\wp^{c_\wp}\mid
H^i_{cris}(X/W(k_\wp))\otimes K_\wp),$$ (see \cite{katz74}).

\setcounter{subsubsection}{\value{equation}\refstepcounter{equation}}

\subsubsection{Semi-simplicity of the crystalline Frobenius}\label{crys-frob}
If $X$ is a $K3$-surface, then it follows from  \cite[Corollaire 1.10]{deligne1981}, \cite{deligne1971} and \cite[2.10 Lemma]{ogus82} that the crystalline Frobenius
$\phi_\wp^{c_\wp}$ is semi-simple.

\setcounter{subsubsection}{\value{equation}\refstepcounter{equation}}

\subsubsection{Mazur's theorem}
\numberwithin{equation}{subsubsection}
We recall from \cite{Mazur1972}, \cite{Mazur1973}, \cite{deligne1981},
\cite{berthelot-ogus} the following theorem of B.~Mazur. There are
two parts to the theorem of Mazur (\cite{Mazur1972,Mazur1973}) that we require, and we record them
separately for convenience. After inverting finitely many primes
$v\in S$ in $K$, we can assume that $X$ has good reduction outside
$S$. Using Proposition~\ref{fintor} (see below) we can assume that
$H^i(X, \Omega^j_X)$ and $H^i_{cris}(X/W(k_\wp))$ are torsion-free outside a
finite set of primes of $K$. As $X$ is defined over characteristic
zero, the Hodge to de Rham spectral sequence degenerates at $E_1$
stage. Thus, all the hypothesis of Mazur's theorem are satisfied. The
two parts of Mazur's theorem that we require are the following:

\subsubsection{Hodge and Newton polygons: Mazur's proof of  Katz's conjecture}\label{mazur}
For the definitions of the Hodge and Newton polygons, see \cite[Section 2]{Mazur1972} or \cite[Chapter 8]{berthelot-ogus}. 
Let $L_\wp$ be a finite extension of field of fractions of
$W(k_\wp)$, over which the polynomial $P_{i,\wp}(t)$ splits into linear
factors. Let $w$ denote a valuation on $L_\wp, $ such that
$w(p)=1$. The Newton polygon of the polynomial $P_{i,\wp}(t)$ lies above
the Hodge polygon in degree $i$,
 defined by the Hodge numbers $h^{j, i-j}$ of degree
$i$. Moreover, they have the same endpoints (see  \cite[Chapter 8, Theorem 8.36]{berthelot-ogus} or \cite[Theorem 1]{Mazur1972}).

\subsubsection{Divisibility} \label{divisibility}
The crystalline Frobenius $\phi_\wp$ is divisible by $p^i$ when
restricted to $$F^iH^j_{dR}(X/W(k_\wp)):= \H^j(X/W(k_\wp)), \Omega^{\geq
i}_X).$$ This is immediate from \cite[Theorem 8.26]{berthelot-ogus} or the main result of \cite[Section A, page 62]{Mazur1973}.

\subsubsection{Crystalline torsion}
We will also need the following proposition (from \cite{joshi00a-arxiv-version}) which is certainly
well-known but as we use it in the sequel, we record it here for
convenience.
\begin{proposition}\label{fintor}
Let $X/K$ be a smooth projective variety. Then for all but finitely
many non-archimedean primes $\wp$, the crystalline cohomology $H^i_{\rm
cris}(X_\wp/W(k_\wp))$ is torsion free for all $i$.
\end{proposition}

\begin{proof}
We choose a model $f:\mathcal{X}\to \Spec(\O_K)$ as in \Cref{ss:models}. By the projectivity of $f$, the relative de Rham cohomology of the model $\mathcal{X}$ is a finitely generated
$\O_K$-module and has bounded torsion.  After removing finite set
$S$ of primes of $\Spec(\O_K)$, we can assume that $H^i_{dR}({\mathcal X}/\O_{K,S})$ is a
torsion-free $\O_{K,S}$ module, where $\O_{K,S}\supseteq \O_K$ is the
ring of $S$-integers in $K$. We can enlarge $S$ by adding to it all the ramified primes of $K$ and hence assume that for all $\wp\not\in S$, one has $\O_{K,\wp}=W(k_\wp)$.

 By the comparison theorem of Berthelot
(see \cite[Chapter 7, Corollary 7.3]{berthelot-ogus}), there is a natural isomorphism of
the crystalline cohomology of $X_\wp$ to that of the de Rham cohomology
of the generic fiber of a lifting to $W(k_\wp)$:
\be\label{eq:comp-crys-derham} H^i_{\rm cris}(X_\wp/W(k_\wp)) \equiv H^i_{dR}({\mathcal X}/
\O_{K,S})\otimes W(k_\wp).\ee
This proves our proposition. 
\end{proof}

\brem
The proof of \Cref{fintor} also
shows that we can assume after inverting some more primes, that the
Hodge filtration terms $F^jH^i_{\rm dR}({\mathcal X}, \O_{K,S})$ are also
locally free over $\O_{K,S}$, such that the sub-quotients are also locally
free modules.
\erem
\numberwithin{equation}{subsection}
\subsection{Proof \Cref{th:positive-density}}
We first note the following lemma which is fundamental to the proof.
\begin{lemma}[Joshi-Rajan] With notation as above, assume the following:
\benumlab
\item if $X$ is a $K3$-surface, then $X$ does not have ordinary
reduction at $\wp$.
\item if $X$ is an abelian variety, then the $p$-rank of the reduction of
$X$ at $\wp$, is at most one.
\eenum
Then $p|a_\wp$.
\end{lemma}

\begin{proof} We argue by contradiction. 
   Let $w$ be a valuation as in \ref{mazur} above. If $X$ is an
abelian variety defined over the finite field $k_\wp$, then the
$p$-rank of $X$ is precisely the number of eigenvalues of the linearized
  crystalline Frobenius (as in Section 3.4) acting on
$H^1(X_\wp/W(k_\wp))\otimes \Q_p$ which are $p$-adic units.  Suppose now
$\alpha$ is an eigenvalue of the crystalline Frobenius
$\phi_\wp^{c_\wp}$ acting on $H^2(X_\wp/W(k_\wp))\otimes \Q_p$. In case (2),
the hypothesis implies that $w(\alpha)$ is positive, and hence
$w(a_\wp)$ is strictly positive. As $a_\wp$ is a rational integer, the
lemma follows.

   When $X$ is a $K3$-surface, it follows from the shapes of the
Newton and Hodge polygons, that ordinarity is equivalent to the fact
that precisely one eigenvalue of $\phi_\wp^{a_\wp}$ acting on
$H^2(X_\wp/W(k_\wp))\otimes \Q_p$ is a $p$-adic unit. Hence, if $X$ is not
ordinary, then for any $\alpha$ as above, we have $w(\alpha_\wp)$ is
positive. Again since $a_\wp$ is a rational integer, the lemma follows.
\end{proof}

Now we complete the proof of \Cref{th:positive-density}.

\begin{proof}[Proof of \Cref{th:positive-density}]
Our proof follows closely the method of \cite{ogus82}.  Let $$d={\rm dim}~H^2_\ell.$$
   Fix a prime $\ell>d$, and let $\rho_\ell$ denote the
corresponding galois representation on $H^2_\ell$. We can choose a finite extension $L/\Q$ such that 
\benumlab
\item $L\supset K$, 
\item $L/\Q$ is Galois, 
\item $L$ contains (all) the $\ell^{th}$-roots of unity, 
\item the absolute Galois group, $G_L$, of $L$  leaves a lattice $V_\ell\subset H^2_\ell$ stable, and the ${\rm mod}$-$\ell$  representation,  denoted by  $\overline{\rho}_\ell$, of $G_L$, on $V_\ell\otimes \Z/\ell\Z$ is trivial i.e. for all $\sigma \in G_L$, $\overline{\rho}_\ell(\sigma)=1$.
\eenum
The property that $\overline{\rho}_\ell(\sigma)=1$ means that the trace of Frobenius at any prime $\wp$ of $L$ satisfies
We have
$$a_\wp\equiv d ({\rm mod}~\ell),$$
where $d={\rm dim}~H^2_\ell$.

The set of primes $\wp$ of $L$ which are of degree one over primes of $\Q$ is well-known to be of density one in $L$. 
Let $\wp$ be a prime of $L$ of degree $1$ over $\Q$, lying over
the rational prime $p$. By this choice, $p$ splits completely in $L$, and as $L$
contains the $\ell^{th}$ roots of unity, we have $p\equiv 1 ({\rm mod}~\ell)$.  Since $|a_\wp|\leq dp$ and is a rational
integer divisible by $p$ from the above lemma, it follows on taking
congruences modulo $\ell$ that
$$ a_\wp=\pm dp.$$

Now $a_\wp$ is the sum of $d$ algebraic integers each of which is
of absolute value $p$ with respect to any embedding. It follows
that all these eigenvalues must be equal, and equals $\pm p$.
Hence, we have that
   $$F_\wp=\pm pI$$
as an operator on $H^2_\ell$. By the semi-simplicity of the crystalline
Frobenius for abelian varieties and $K3$-surfaces (see \Cref{crys-frob}), it
follows that $\phi_\wp=\pm pI$ (for a positive density of primes $\wp$ of $L$). But this contradicts the divisibility
property of the crystalline Frobenius (\ref{divisibility}), that the
crystalline Frobenius is divisible by $p^2$ on $F^2H^2_{\rm cris}(X_\wp/W(k_\wp)) \simeq F^2H^2_{dR}({\mathcal X}/
\O_{K,S})\otimes W(k_\wp)$.

Hence, at least one eigenvalue of Frobenius on $H^2_\ell$ is a $p$-adic unit. So if $X$ is a K3 surface then $\wp$ has to be a prime of ordinary reduction, and if $X$ is an abelian variety then $X$ must have $p$-rank at least two. Since $\wp$ is a prime of degree one in $L$, one knows that such primes are of density one in $L$ and hence this proves all the assertions.
\end{proof}

\subsection{Ordinary reductions for varieties of K3 type}
This method can  be axiomatized to give positive
density results whenever certain cohomological conditions are
satisfied. We present this formulation for the sake of completeness.

\begin{proposition}[Joshi-Rajan]
Let $X$ be a smooth projective variety over a number field
$K$. Suppose $M$ is a polarized motive with \'etale realization $H_{\acute{e}t}(M,\Q_\ell)$ given by $H^i_{\acute{e}t}(X,\Q_\ell)$ for some $i>0$ (and any prime $\ell$). Suppose $M$ is of K3 type. This means that the following conditions are satisfied:
\benumlab
\item the Hodge structure $H_{dR}(M)$ is of K3 type, in particular, the Hodge number  $h^{0,2}(M)=1$,
\item The action of the crystalline Frobenius on the crystalline realization $H_{cris}(M_\wp)$
of $M$ at the prime $\wp$  is semi-simple for all
but finite number of primes $\wp$ of $K$.
\eenum
 Then the representation $H_{\acute{e}t}(M,\Q_\ell)$ is
ordinary at a set of primes of positive density in $K$ and the
$F$-crystal $H_{cris}(M_\wp)$ (given by the crystalline realization of $M$) is ordinary for these primes. In
other words, the motive $M$ has ordinary reduction for a positive
density of primes of $K$.
\end{proposition}
\brem 
This may be applied to varieties of K3 type where {\bf(1,2)} are available. For example, this is the case for cubic fourfold $X\subset \P^5$  over a number field $K$. In this situation \cite{rapoport72} establishes {\bf(1)} and semi-simplicity of the action of Frobenius on $H_{\acute{e}t}(M,\Q_\ell)$ and using this, we can establish {\bf(2)}  along the lines of \cite{deligne1981} (also see \cite{levin2001}). In particular, any smooth,projective  cubic fourfold over a number field has good ordinary reduction at a positive density of primes of $K$. 
\erem

\section{Primes of Hodge-Witt reduction}\label{se:hodge-witt-reduction}
\subsection{Hodge-Witt reduction}

\begin{definition}
Let $X/K$ be a smooth, projective variety over a number field. Let $\mathfrak{X}\to\Spec(\O_K)$ be a model \Cref{ss:models}. We say that a $X$ has good Hodge-Witt reduction at $\wp$ if the fiber  $X_\wp$ of $\mathfrak{X}$ at $\wp$ is a Hodge-Witt variety \Cref{def:hw}.
\end{definition}

Since ordinary varieties are Hodge-Witt, we can formulate a weaker
version of Conjecture~\ref{serre}.

\begin{conj}[Joshi-Rajan]\label{con:positive-density-hodge-witt}
Let $X/K$ be a smooth projective variety over a number field
then $X$ has Hodge-Witt reduction modulo a set of primes of $K$ of
positive density.
\end{conj}

\br
Let us understand what our conjecture means in the context of a K3 surface. Let $X/K$ be a smooth, projective K3 surface over a number field $K$.  So Conjecture~\ref{con:positive-density-hodge-witt} predicts in this case that there are infinitely many primes of Hodge-Witt reduction for $X$. Now $X$ has Hodge-Witt reduction at a prime $\wp$ of $K$ if and only if $X$ is of finite height (i.e. the formal group $\Phi^2$ attached to $X$ by \cite[\S 3]{artin77} is of finite height).
\er

As we remarked in the Introduction, \Cref{con:positive-density-hodge-witt} and Conjecture~\ref{serre} are not equivalent (see Example~\ref{elliptic-curve-example}). The two however are related and the following clarifies the relationship between Conjecture~\ref{serre} and \Cref{con:positive-density-hodge-witt}. Our theorem is the following:
\bthm\label{th:equivalence}
Let $X$ be a smooth, projective variety over a number field $K$. Then the following are equivalent:
\begin{enumerate}
\item There are infinitely many primes of ordinary reduction for $X$.
\item There are infinitely many primes of ordinary reduction for $X\times_K X$
\item There are infinitely many primes of Hodge-Witt reduction for $X\times_K X$.
\end{enumerate}
\ethm
\begin{proof}
Let us equip ourselves with models for $X$ and $X\times_K X$ over a Zariski-open subscheme of $U\subset\spec(\O_K)$ with smooth, proper, fibers over $U$. Let $\wp\in U$ be a prime of good, ordinary reduction for $X$. As a product of ordinary varieties is ordinary (see \cite{illusie90a}), so $X\times_K X$ is also ordinary at $\wp$. Thus (1) $\implies$ (2). Since any prime $\wp$ of good, ordinary reduction for $X\times_K X$ is also a prime of good, Hodge-Witt reduction for $X\times_KX$, so we have (2) $\implies$ (3). Let us prove that (3) $\implies$ (1). Assume $\wp$ is a prime of Hodge-Witt reduction for $X\times_K X$. By \cite[III, Prop 2.1(ii) and Prop 7.2(ii)]{ekedahl85} we see that a product is Hodge-Witt if and only if at least one factor is ordinary and other is Hodge-Witt. Thus, if $X\times_K X$ is Hodge-Witt at $\wp$ then $X$ must be ordinary at $\wp$. Thus, $X$ is ordinary at $\wp$.
\end{proof}

\subsection{Hodge-Witt Abelian varieties} Let $X$ be a self-dual $p$-divisible group over a perfect field $k$ of characteristic $p>0$ (we make the self-duality assumption because we work here in the context of abelian varieties). Let $X^t$ denote the Cartier dual of $X$ and let $X_{(0,1)}$ be the local-local part of $X$. Following \cite[Def 3.4.2]{conrad-book} we say that a $p$-divisible group $X$ is of \textit{extended Lubin-Tate type} if $\dim X_{(0,1)}\leq 1$ or $\dim X^t_{(0,1)}\leq 1$. Equivalently $X$ is of extended Lubin-Tate type if $X$ isogenous to an ordinary $p$-divisible group  or $X$ is isogenous to a product of an  ordinary $p$-divisible group and one copy of the $p$-divisible group of a supersingular elliptic curve. We begin with the following proposition:
\bpro\label{thm:cm-hodge-witt-3}
Let $k$ be a perfect field of characteristic $p>0$. Let $A$ be an abelian variety over $k$. Then the following are equivalent:
\benum
\item $A$ is Hodge-Witt.
\item $A$ has $p$-rank $\geq \dim(A)-1$.
\item The slopes of Frobenius of $A$ are in $\{0,1,\frac{1}{2}\}$ and multiplicity of $\frac{1}{2}$ is at most two.
\item The $p$-divisible group $A[p^\infty]$ is of extended Lubin-Tate type.
\eenum
\epro
\bp
The equivalence (1)$\iff$(2) is \cite[Corollary 6.3.16]{illusie83a}. The equivalence (2)$\iff$(3)$\iff$(4) is immediate from the definition.
\ep

\br
In \cite[Paragraph after 3.4.2 Definition]{conrad-book} an abelian variety $A$ with $p$-rank equal to $\dim(A)-1$ is called an \textit{almost ordinary abelian variety}. Thus, we see that $A$ is Hodge-Witt if and only if $A$ is ordinary or $A$ is almost ordinary.
\er

\br
As should already be clear from the previous remark, properties of Hodge-Witt varieties and non-Hodge-Witt varieties can be rather distinct. But subtler examples of the difference between the two exist and are implicit in \cite{conrad-book}. Since the phrase \textit{Hodge-Witt} is never mentioned in \cite{conrad-book}, in \Cref{thm:trans-conrad-book}, we provide the following transliteration of the relevant results.
\er

\bthm[Reformulation of {\cite[Theorem 3.5.1, Corollary 3.5.6 and Remark 3.5.7]{conrad-book}}]\label{thm:trans-conrad-book}
Let $k$ be a perfect field of characteristic $p>0$, let $A$ be an abelian variety with a CM algebra $K$ with $\dim(K)=2\dim(A)$ and an embedding $K\into \End^0_{\bar k}(A)=\End^0_k(A)$. 
\benumlab
\item If $A$ is Hodge-Witt, then there exists an isogeny $A'\to A$ such that $A'$ admits a CM-lifting to characteristic zero. \item On the other hand, if $A$ is not Hodge-Witt, there exists an isogeny $A'\to A$, such that $A'$ does not admit any CM-lifting to characteristic zero.
\eenum
\ethm

\subsection{Primes of Hodge-Witt reduction for abelian threefolds}
Our next theorem shows that
\Cref{con:positive-density-hodge-witt} is true for abelian
varieties of dimension at most three. This result is implicit in \cite{joshi00a-arxiv-version}.

\begin{theorem}[Joshi-Rajan]\label{th:positive-density-hodge-witt-abelian-threefold}
Let $A/K$ be an abelian variety of dimension at most three over a
number field $K$. Then there exists a finite extension $L/K$ and a
set of primes of density one in $L$ at which $A$ has Hodge-Witt
reduction.
\end{theorem}

\begin{proof}
The case $\dim(A)=1$ is immediate as smooth, projective curves are
always Hodge-Witt by Proposition~\ref{pr:curves-are-hodge-witt} and hence one may take $L=K$. Next assume that $\dim(A)=2$.  By \cite[2.9 Corollary]{ogus82}, there exists finite extension $L/K$, and a density one set of primes of $L$, all which are of degree one in $L$,  such that $A$ has ordinary reduction at any prime (of $L$) in this set. So we are
done in this case.

Thus, we need to address $\dim(A)=3$. This is clear by Theorem~\ref{th:positive-density}(2) and the characterization of Hodge-Witt abelian varieties provided by Proposition~\ref{thm:cm-hodge-witt-3}.  This proves  the result.
\end{proof}

\subsection{Hodge-Witt reductions of some surfaces and all smooth Fano threefolds}
For surfaces the geometric genus appears to detect the size of the
set of primes which is predicted in
\Cref{con:positive-density-hodge-witt}.

\begin{theorem}[Joshi-Rajan]\label{pgtheorem}
   Let $X$ be a smooth, projective surface with $p_g(X)=0$,
defined over a number field $K$.  Then \Cref{con:positive-density-hodge-witt} is true for $X$ i.e. for all but finitely many
primes $\wp$ of $K$, $X$ has Hodge-Witt reduction at $\wp$.
\end{theorem}

\begin{proof}
To prove that a smooth, proper surface (over a perfect  field of characteristic $p>0$) is Hodge-Witt, it is sufficient to prove that $H^2(X,W(\O_X))$ is of finite type. This is a consequence of the fact that under the hypothesis there is at most one non-zero differential in the slope spectral sequence of such a surface, namely $d:H^2(X,W(\O_X))\to H^2(X,W\Omega_X^1)$. Note that all $H^i(X,W\Omega^j_X)$  except for the cohomology groups appearing in this differential are of finite type. To see this  see \cite[Proof of Theorem 2.4]{nygaard79b} or \cite[Corollary 3.3, Theorem 3.7, Proposition 3.11, Proposition 3.12]{illusie79b} and this differential is zero if $H^2(X,W(\O_X))$ is of finite type. 	

In fact, we claim the following: let  $Y/k$ be smooth proper surface with geometric genus $p_g(Y)=0$ over a perfect field of characteristic $p>0$.  Then  $$H^2(Y,W(\O_Y))=0.$$ To prove this one uses the exact sequence (for all $n\geq1$, and $W_{n}(\O_Y)$ is the sheaf of Witt-vectors of length $n$)
$$\xymatrix{0\ar [r]& W_n(\O_Y)\ar [r]^V & W_{n+1}(\O_Y) \ar [r] &\O_Y\ar [r] &0,}$$
and one proves, by induction on $n$, that $H^2(Y,W_n(\O_Y))=0$ starting with the $n=1$ case $H^2(Y,\O_Y)
=0$ (i.e. $p_g(Y)=0$) and the fact that $\dim(Y)=2$ and hence on passage to the limit one sees that $H^2(Y,W(\O_Y))=0$. 

Now by choose a model  for $X/K$ and shrinking the base of required one can assume that all the fibers $X_\wp$ satisfy $p_g(X_\wp)=0$. Then by taking $Y=X_\wp$ in the above one obtains that $$H^2(X_{\wp},W(\O_{X_\wp}))=0$$   for all but a finite number of
primes $\wp$ of $K$. So $X_\wp$ is Hodge-Witt and hence for all but finite number of primes $\wp$ of $K$ as claimed.
\end{proof}

\begin{corollary}[Joshi-Rajan]
Let $X/K$ be an Enriques surface over a number field $K$. Then $X$ has
Hodge-Witt reduction modulo all but finite number of primes of $K$.
\end{corollary}

\begin{proof}
This is immediate from Theorem~\ref{pgtheorem} as for any Enriques surface $X$ over $K$, one has
$p_g(X)=0$.
\end{proof}

When $X$ is a smooth Fano surface over a number field, one can prove a
little more:

\begin{theorem}[Joshi-Rajan]\label{fanotheorem}
Let $X$ be a smooth, projective Fano surface, defined over a number
field $K$.  Then for all but finitely many primes $\wp$, $X$ has
ordinary reduction at $\wp$ and moreover the de Rham-Witt cohomology of
$X_\wp$ is torsion free.
\end{theorem}

\begin{proof}
     By \cite[Exercise 1.6.E(5)]{shrawan-book} one knows 
that if $X$ is a smooth, projective and Fano variety $X$ over a
number field, then for all but finitely many primes $\wp$, the reduction
$X_{\wp}$ is Frobenius split.  By \cite[Proposition 3.1]{joshi03}
we see that the reduction modulo all but finitely many primes $\wp$ of $K$
gives an ordinary surface. Then by \cite[Lemma 9.5]{bloch1986} and
Proposition~\ref{fintor}, the result follows.
\end{proof}

\begin{example}[Joshi-Rajan] \label{fano} Let $K=\Q$ and
$X\subset \P^n$ be any Fermat hypersurface of degree $m$ and $n\geq
6$. If $m<n+1$ then this hypersurface is Fano, but by \cite[\S1, Theorem]{toki96}
this hypersurface does not have Hodge-Witt reduction at primes $p$
satisfying $p\not\equiv 1\mod m$. This gives examples of Fano
varieties which are Frobenius split but are not Hodge-Witt or ordinary.
\end{example}

\begin{remark}[Joshi-Rajan]
It is clear from Example~\ref{fano} that there exist Fano varieties
over number fields which have non-Hodge-Witt reduction modulo an
infinite set of primes and thus this indicates that in higher
dimension $p_g(X)$ is not a good invariant for measuring this
behavior. 
\end{remark}

\begin{question}[Joshi-Rajan]\label{hodge-level} The notion of Hodge level is introduced in \cite{deligne72} and complete intersections of Hodge level at most one are classified in \cite{rapoport72}. From this classification one sees that complete intersections of Hodge level one are all Fano varieties. Hence one can ask the following: let $X/K$ be a smooth, projective Fano variety over a number
field. Assume that $X$ has Hodge level $\leq 1$ in the sense of
\cite{deligne72,rapoport72}. Then does $X$ have Hodge-Witt reduction
modulo all but a finite number of primes of $K$?
\end{question}

\begin{remark}[Joshi-Rajan]
   A list of all the smooth complete intersection in $\P^n$ which
are of Hodge level $\leq 1$ is given in \cite[Table 1]{rapoport72}. By \cite[Th\'eor\`eme 2.2]{illusie90a}
a general, smooth complete intersection of dimension $n$ in $\P^{n+d}$ is ordinary and hence 
general, smooth complete intersections of Hodge level $\leq 1$ given in \cite[Table 1]{rapoport72} are Hodge-Witt. In the notation of \cite{rapoport72} one can verify, using \cite{suwa93}, that all smooth complete intersections Hodge level one are Hodge-Witt except possibly $V_n(2,2), V_n(2,2,2)$. It seems reasonable to expect that the result is true even for these cases at least for large characteristics.
\end{remark}

\begin{remark}
In  \cite{joshi07} we have  answered the Question~\ref{hodge-level}
affirmatively  for $\dim(X)=3$ where the Hodge
level condition is automatic.
\end{remark}

\bthm\label{th:fano-threefold}
Let $X/K$ be a smooth, projective and Fano variety of dimension three over a number field $K$. Then $X/K$ has Hodge level at most one and $X$ has Hodge-Witt reduction modulo all but finitely many primes of $K$.
\ethm
\bp 
This uses methods of \cite{joshi04b}. From \cite{rapoport72}, $X$ has Hodge level at most one if 
$$\sup\{|i-j|: H^{j}(X,\Omega_X^i)\neq 0\}\leq 1.$$
Since $X$ is Fano,  and dimension three, one has $H^j(X,\O_X)=0$ for $1\leq j\leq 3$ and the above inequality evidently holds.  By Hodge symmetry
$H^j(X,\O_X)=0=H^0(X,\Omega_X^j)$ for $1\leq j\leq 3$. Now choose a model $\mathcal{X}$ for $X$ over $\Spec(\O_K)$  (\Cref{ss:models}). Then there exists a non-empty open subset $U=\Spec(\O_K)-V(I)$ such that these vanishings  hold  for the special fiber $X_\wp$ for all $\wp\in U$. The vanishing of $H^j(X_\wp,O_{X_\wp})=0$ for $1\leq j\leq3$ gives, using the method of proof of \Cref{pgtheorem}, that $H^j(X_\wp,W(\O_{X_\wp}))=0$ for $1\leq j\leq 3$. This means that the de Rham-Witt differential $$H^j(X_\wp,W(\O_{X_\wp}))\to H^j(X_\wp,W\Omega_{X_\wp}^1)$$
is zero. From \cite{joshi04b}, one sees that the only differential which need be considered are
$$H^2(X_\wp,W\Omega^1_{X_\wp})\to H^2(X_\wp,W\Omega_{X_\wp}^2)$$
and
$$H^3(X_\wp,W\Omega^1_{X_\wp})\to H^3(X_\wp,W\Omega_{X_\wp}^2).$$
These are zero by the duality argument of \cite{joshi04b}. This proves that $X_\wp$ is Hodge-Witt.
\ep

\newcommand{\bQ}{\overline{\Q}}

\subsection{Primes of Hodge-Witt reduction for CM abelian varieties}
We will recall a few facts about CM abelian varieties. For a modern reference for CM abelian varieties reader may consult \cite{conrad-book}.
A \emph{CM field $K$} is a finite extension of $K/\Q$ which has no real embeddings and which contains a totally real subfield $K^+\subset K$ with $[K:K^+]=2$. 

Recall from
\cite[Definition 1.3.3.1]{conrad-book} that a \emph{CM algebra $P$} is a product $$P=K_1\times K_2\times\cdots\times K_s$$ of finitely many CM fields $K_1,\ldots,K_s$.

Let $L$ be a number field. For an abelian variety $A/L$, write $\End^0_L(A)=\End(A)\tensor_{\Z}\Q$. Following \cite[Definition 1.3.8.1]{conrad-book}, by an \emph{abelian variety $A/L$ with complex multiplication} (i.e. a \emph{CM abelian variety $A/L$}), we mean 
\benumlab
\item an abelian variety $A/L$, and
\item an embedding $j:P\into \End_L^0(A)=\End_L(A)\tensor_{\Z}\Q$, of some CM algebra $P$, such that
\item  $\dim_{\Q}(P)=2\dim(A)$. 
\eenum

From \cite[Proposition 1.3.2.1]{conrad-book} and \cite[Theorem 1.3.4]{conrad-book}, it is immediate that if $A/L$ is an abelian variety with CM, then any $L$-isotypic factor $A_j$ of $A$ also has CM and any $L$-simple factor of $A/L$ also has CM.

\bthm\label{th:CM-ab-general} 
Let $L$ be a number field and let $A/L$ be an abelian variety with CM. Then $A/L$ has infinitely many primes of ordinary and Hodge-Witt reduction. In general the two sets of primes may not coincide.
\ethm

This theorem will be proved by reduction to the case of simple abelian varieties with CM. This special case is established in \Cref{thm:cm-hodge-witt-2}, and general case (\Cref{th:CM-ab-general}) will be established in \Cref{ss:proof-th:CM-ab-general}.

\subsection{Hodge-Witt reduction of simple abelian varieties with complex multiplication}
To complete the proof of \Cref{th:CM-ab-general}, we need to establish the case of simple abelian varieties with complex multiplication. This is carried out in this section. We will recall a few more facts about simple abelian varieties with CM. For a modern reference for simple CM abelian varieties reader may consult \cite{conrad-book}. 

Let $K$ be a CM field. Let $L$ be a number field and let $A$ be an abelian variety over $L$ satisfying
\benumlab
\item $A$ is absolutely simple, and
\item $\End^0_L(A)=K$, and
\item such that $[K:\Q]=2\dim(A)$.
\eenum

\emph{In this paper, by a simple abelian variety with CM, we will always mean an abelian variety of this type.} In particular, such an abelian variety is always simple.

By passing to a finite extension if required, we will assume (in addition to {\bf(1)--(3)}) that $L$ contains
 \benum[start=4,label={\bf(\arabic{*})}]\label{cm-assumptions}
 \item
 the galois closure $N$ of $K$
 \item and that $A$ has good reduction at all finite primes of $L$.
 \eenum
Assumption {\bf(5)} is guaranteed by the theory of Complex multiplication (see \cite[1.7.2.3 Theorem]{conrad-book}). From now on assume that the assumptions {\bf(1)--(5)} are hold.

Let $\Phi$ be the CM-type of $A$. Let $E(K,\Phi)$ be the reflex field. In this situation by \cite{yu04} we know that the Dieudonne module of the reduction $\bar{A}$ at any finite prime $\wp$ of $L$ is determined by $\Phi$ and the splitting of rational primes in $K$.

In particular, it is possible to completely determine the reduction type of $A$ from the decomposition of $p$  in $\O_K$. But in general this is a complicated combinatorial problem (see \cite{goren10,sugiyama14,zaytsev13}) depending on the galois closure of $K$ and also on the CM-type (readers will find \Cref{ex:hodge-witt-cm-example} useful). In the present section we give sufficient conditions on prime decomposition of $p$ which allows us to conclude the infinitude of primes of Hodge-Witt and ordinary reduction for $A$. Before proceeding we make some elementary observations.

\begin{definition}\label{def:almost-completely}
Let $K/\Q$ be a finite extension. Let $p$ be a prime number which is unramified in $K$. We say that $p$ \textit{splits almost completely} in $K$ if we have a prime factorization:
$$
(p)=\prod_{i=1}^n\wp_i,
$$
with $n\geq [K:\Q]-1$ if $[K:\Q]>2$ and $n=[K:\Q]$ if $[K:\Q]\leq 2$. We say that $p$ \textit{splits completely} in $K$ if we have a prime factorization:
$$
(p)=\prod_{i=1}^n\wp_i,
$$
with $n=[K:\Q]$.
Let $f_i$ be the degree of the residue field extension of $\wp_i$ over $\Z/p$. Then the \textit{residue field degree sequence} of $p$ in $K$ is the tuple $(f_1,f_2,\ldots)$.
\end{definition}

It is clear from Definition~\ref{def:almost-completely} that if $p$ splits completely in $K$ then $p$ splits almost completely in $K$. Also from our definition $p$ splits almost completely in a quadratic extension if and only if $p$ splits completely. The following proposition is an elementary consequence of Definition~\ref{def:almost-completely} and standard facts about splitting of primes in $K$, but we give a proof for completeness.
\bpro\label{p:almost-completely-split}
Suppose $K$ is a CM field with $[K:\Q]=2g$. Let $K_0\subset K$ be the totally real subfield of $K$. Then we have the following assertions:
\benumlab
\item If $p$ splits completely in $K$ then  $p$ splits completely in $K_0$.
\item If $p$ splits almost completely, but not completely in $K$ then $p$ splits completely in $K_0$ and exactly one prime lying over $p$ in $K_0$ is inert in $K$ and the rest split in $K$.
\item  If $K/\Q$ is  galois then $p$ splits almost completely in $K$ if and only if $p$ splits completely in $K$.
\eenum
\epro

\bp
Let $[K:\Q]=2g$. Let 
\be 
(p)=\wp_1\cdots \wp_n
\ee
be the prime factorization of $p$ in $K$ (recall that by definition $p$ splits completely means $p$ is unramified) and let
\be 
(p)=\wq_1\cdots \wq_\ell,
\ee 
be the prime factorization of $p$ in $K_0$. Let $m$ be the number of primes in $\wq_1,\ldots,\wq_\ell$ which split in $K$. By renumbering the $\wp_1,\ldots,\wp_n$ we can assume that $\wp_1,\ldots,\wp_{2m}$ are the primes lying over the split primes $\wq_1,\ldots,\wq_m$. The rest lie over primes of $K_0$ (over $p$) which are inert in $K$. Thus we have
\be 
2g=2m+(\ell-m)
\ee
equivalently 
\be 
2g=\ell+m.
\ee
Further if $f(\wp|p)$ denotes the degree of the residue field of a prime $\wp$ lying over $p$ then we have for $K_0$ the equation
\be 
\ell \leq \sum_{i=1}^\ell f(\wq_i|p)=g.
\ee
Thus, we see that $2g=\ell+m\leq 2\ell\leq 2g$ as $m\leq \ell$. Hence, we have $2g\leq 2\ell\leq 2g$ so $\ell=g$. Thus, $p$ splits completely in $K_0$. This proves the first assertion.

Now suppose $p$ splits almost completely in $K$ but not completely in $K$. Then $[K:\Q]=2g>2$ by definition (see \ref{def:almost-completely}) as $p$ splits almost completely for $[K:\Q]=2$ if and only if $p$ splits in $K$. So if $p$ is a such a prime then $n=2g-1>1$ which gives $g>1$.  Further, using the  notation established in the previous case we have
\be 
2m+(\ell-m)=\ell+m=2g-1.
\ee
Thus, $\ell+m=2g-1$. We claim that $m\geq g-1$. Suppose $m<g-1$ then 
\be 
2g-1=\ell+m <\ell +g-1,
\ee  
so $g<\ell$. But as $\ell \leq \sum_{i=1}^\ell f(\wq_i|p)=g$ one gets $\ell\leq g$. So $g<\ell$ and $\ell \leq g$ which is a contradiction. This gives $m\geq g-1$. If $m=g$ then we are in the completely split case. So $m=g-1$. Then $2g-1=\ell+m=\ell+g-1$ and hence $\ell=g$. Further as $g>1$ so $m=g-1>1$ and hence exactly one prime of $\wq_1,\ldots,\wq_g$ is inert in $K$ as claimed. This proves the second assertion.

Finally suppose $K/\Q$ is galois and $p$ splits almost completely then $p$ is completely split as all the residue field degrees are equal, while  from the preceding discussion exactly one residue field degree  of prime lying over $p$ in $K$ is two if $p$ splits almost completely but not completely split, so we are done by the second assertion.
\ep 
The following is also immediate from this proof.
\bcor\label{cor:split-type}
Let $K$ be a CM field. Let $p$ be a rational prime which is unramified in $K$. Then the following are equivalent:
\benum
\item $p$ splits almost completely in $K$.
\item the residue field degree sequence of $p$ in $K$ is $(1,1,1,\cdots)$ or $(2,1,1,1\cdots)$.
\eenum
\ecor

We are now ready to prove our theorem. 
\bthm\label{thm:cm-hodge-witt}
Let $A,L,K$ be as above with $K$ a CM field (so that $A/L$ is a simple abelian variety with CM by $K$). Let $\wp_L$ be a prime of $L$ lying over an unramified prime $p$ of $\Q$.  If $p$ splits almost completely in $K$,  then $A$ has Hodge-Witt reduction at $\wp_L$.
\ethm

\bp
\newcommand{\bkappa}{\bar{\kappa}}
We follow a method due to \cite{goren10,sugiyama14,zaytsev13}--especially in \cite[Theorem 1.2]{sugiyama14} the result is proved when $p$ is completely split (in this case $A$ has ordinary reduction at $\wp_L$); but also see \cite{goren10,zaytsev13}. We note that \cite{goren10} spells out the details quite well (for abelian surfaces--and it is also enough to prove what we need here). Suppose $\wp_L$ is a prime lying over $p$ in $L$. Let $\kappa$ be the residue field of $\wp_L$. Then the reduction of $A$ at $\wp_L$ is a abelian scheme over $\kappa$. It will be convenient to extend the base field of the reduction from $\kappa$ to an algebraic closure $\bkappa$ of $\kappa$. Let $X$ be the $p$-divisible group of this abelian scheme over $\bkappa$.   Let $\O_K$ be the ring of integers of $K$, let $K_0\subset K$ be the totally really subfield and $\O_{K_0}$ be the ring of integers of $K_0$. Let $\sigma:\O_K\to\O_K$ be complex conjugation. Then $\O_K$ operates on the $p$-divisible group $X$.

Thus, we are in the following situation: we have a $p$-divisible group $X$ over $\bkappa$ with CM by $K$ and we have to show that  if $p$ splits almost completely in $K$ then $X$ is of extended Lubin-Tate type.

Let $M=M(X)$ be the Dieudonne module of $X$. Then $M$ is a $W=W(\bkappa)$-module and
let us note, by the standard theory of complex multiplication (see \cite[1.4.3.9 Proposition]{conrad-book}), that $M$ is an $\O_K$-module of rank one and hence, or at any rate, $M$ is an $\O_K\tensor\Z_p$-module. We may decompose $\O_K\tensor\Z_p$ using the factorization of $(p)$. This also gives us a decomposition of the Dieudonne module $M=M(X)$ of $X$.

Assume that $[K:\Q]=2g$ and $(p)=\prod_{i=1}^{n}\wp_i$. Then
	$$\O_K\tensor\Z_p=\prod_{i=1}^n W(k(\wp_i))$$
where $k(\wp_i)$ is the residue field of $\wp_i$.
Then we have to show that if $n\geq 2g-1$ then $X$ is of extended Lubin-Tate type.

If $n=2g$ then $p$ splits completely and we are done by \cite{sugiyama14}, but we recall his argument here as it is also needed in case $p$ splits almost completely. If $p$ is completely split, then by Proposition~\ref{p:almost-completely-split}(1) every $\wp_i$ pairs uniquely with a prime $\wp_j$ such that $\wp_j=\sigma(\wp_i)$ (i.e $\wp_i,\wp_j$ live over a prime of $\O_{K_0}$ which splits completely in $\O_K$). The factor of $M$, corresponding to $\wp_i,\wp_j$ provides a Dieudonne module of a $p$-divisible group, which descends to the residue field of the unique prime of $\O_{K_0}$ which splits into $\wp_i,\wp_j$ in $\O_{K}$, gives a $p$-divisible group of type $G_{1,0}\times G_{0,1}$ (in the standard notation of $p$-divisible groups) with slopes $0,1$. If $p$ is completely split, then there are exactly $g$ such factors and hence $X$ is ordinary $p$-divisible group and hence $A$ has ordinary reduction at $\wp_L$. If $p$ splits almost completely  but does not split completely  then by Proposition~\ref{p:almost-completely-split}(2) all but one of $\wp_i$ can be paired with exactly one $\wp_j$ as above and there are exactly $2g-2$ such primes, while the remaining prime $\wp_{2g-1}$ (after renumbering if required), lives over the unique prime of $\O_{K_0}$ lying over $p$ which is inert in $K$. As is shown in \cite{sugiyama14}, the factor of $M$ corresponding to $\wp_{2g-1}$ gives an indecomposable Dieudonne module of rank two over the residue field of $\wp_{2g-1}$ with slope $\frac{1}{2}$, and hence a $p$-divisible group $G_{1,1}$. Thus, in this case $X$ is extended Lubin-Tate group of height $2g$. This completes the proof of the theorem.
\ep

\bthm\label{thm:cm-hodge-witt-2}
Let $A,L,K$ be as in \Cref{thm:cm-hodge-witt}. Then there exists infinitely many primes of Hodge-Witt reduction. In particular, \Cref{con:positive-density-hodge-witt} is true for abelian varieties with complex multiplication and in general the density of such primes is greater than the density of primes of ordinary reduction.
\ethm

\bp
This is a consequence of Theorem~\ref{thm:cm-hodge-witt} and the Chebotarev density theorem. Let $N$ be the galois closure of $K$ in $\bar{\Q}$, let $G=\gal(N/\Q)$. Let $H\subset G$ be the subgroup fixing $H$. Then action of $G$ on the coset space $G/H$ embeds $G\into S_n$ where $n=[G:H]$  and In particular, $\frob_p\in G$ acts on $G/H$ (by permutation). This action depends only on the conjugacy class of $\frob_p$.  It is well-known, see for instance \cite{lenstra} and \cite{Lang1970} that the splitting type is determined by the cycle decomposition of $\frob_p$ on $G/H$ and that the splitting type of $p$ in $K$ is determined by the sequence of residue field degrees. In particular, we can ``read off'' all the information we need from the table of conjugacy classes of $G$. By Corollary~\ref{cor:split-type} we see that $p$ splits completely if and only if the cycle decomposition is identity, and the $p$ splits almost completely if the cycle decomposition is identity or exactly one transposition. Let $G_{tr}$ be the union of the conjugacy classes of these two types. Then $G_{tr}\supseteq \{1\}$ and hence $G_{tr}\neq \emptyset$. In particular, we see that the density of primes of Hodge-Witt reduction is $\frac{|G_{tr}|}{|G|}\geq \frac{1}{|G|}>0$ and this proves the theorem. To prove the last assertion it is enough to give an example. See Example~\ref{ex:hodge-witt-cm-example} below.
\ep

\bcor\label{cor:cm-hodge-witt-densities}
Let $A,L,K$ be as in \Cref{thm:cm-hodge-witt} with
$K$ a CM field with galois closure $N/\Q$ and $G=\gal(N/\Q)$. Let $G_{tr}$ be as above. Then the density of primes of Hodge-Witt reduction for $A$ is at least $\frac{|G_{tr}|}{|G|}$ and the density of primes of ordinary reduction is at least $\frac{1}{|G|}$.
\ecor
\bp 
This is immediate from the proof of \Cref{thm:cm-hodge-witt-2}.
\ep
\subsection{Proof of \Cref{th:CM-ab-general} and some of its corollaries}\label{ss:proof-th:CM-ab-general}
Before proving \Cref{th:CM-ab-general}, it will be convenient to extend \Cref{def:almost-completely}, and some of the properties established earlier for a CM field, to the case of a general CM algebra.

Let $P=K_1\times K_2\times\cdots\times K_s$ be a CM algebra. Let $p$ be a rational prime. One says that \emph{$p$ splits completely in $P$} if $p$ splits completely in ${K_j}$ for each $j=1,\ldots,s$. One says that \emph{$p$ splits almost completely in $P$} if $p$ splits in all but one of the factors ${K_j}$ of $P$ and $p$ splits almost completely in the remaining factor. 

\brem
Note that it follows the above definition and from \Cref{def:almost-completely} that if all the factors $K_j$ of $P$ are quadratic fields, then $p$ splits almost completely in $P$ if and only if $p$ splits completely in $P$. \Cref{ex:ab-fourfold2} will help clarify the subtleties which occur.
\erem

\blem\label{le:density-in-prods} 
If $P=K_1\times K_2\times \cdots \times K_s$ is a CM algebra, then there exists a positive density of rational primes $p$ which  split almost completely in $P$. In general, the set of rational primes which split completely in $P$ need not coincide with the set of primes which split almost completely in $P$.
\elem
\bp 
This is proved in two steps. For $s=1$, this is proved in the course of the proof of \Cref{thm:cm-hodge-witt-2}.  If $p$ splits completely in each $K_j$, then $p$ splits completely in $P$ and hence $p$ splits almost completely in $P$. Thus to prove the theorem, it is enough to prove that there exist a positive density of primes $p$ which split completely $P$ equivalently $p$ splits completely in each $K_j$. 

To do this, fix an algebraic closure $\bQ$ of $\Q$ and view each $K_j\subset \bQ$. Let $K=K_1K_2\cdots K_s\subset \bQ$ be the compositum of all the $K_j$. Then by \cite[Chapter VIII, \S4]{Lang1970}, there exists a positive density of rational primes $p$ which split completely in $K$ and every such prime splits completely in each $K_j$.  Thus, there exists a positive density of rational primes which split completely in $P$ and hence split almost completely in $P$.

So it remains to show that the two sets of primes of ordinary and Hodge-Witt reductions need not coincide in general. If $P=K_1$ is a CM field, then this is immediate from the proof of \Cref{thm:cm-hodge-witt-2}, \Cref{ex:hodge-witt-cm-galois-example} and \Cref{ex:hodge-witt-cm-example}. So let me give examples where $P$ is not a field. For the CM algebra $P$ considered in \Cref{ex:ab-fourfold} (which was suggested by the Referee),  the two sets of rational primes asserted by \Cref{le:density-in-prods} concide; while for the CM algebra considered in \Cref{ex:ab-fourfold3}, the two sets of primes do not coincide. This completes the proof. 
\ep

\bp[Proof of \Cref{th:CM-ab-general}]
This theorem will be proved by reducing to the case of a simple abelian variety with CM established in \Cref{thm:cm-hodge-witt-2}. This reduction is carried out as   follows. Suppose $A=A_1\times A_2\times \cdots \times A_s$ is a decomposition of $A$ into pairwise non-isomorphic $L$-isotypic factors. Choose an isotypic factor $A_j$ of $A$. Let $A_j=B^{m_j}$ for some simple abelian variety $B/L$ and an integer $m\geq 1$. Then by \Cref{th:equivalence},  $A_j$ has infinitely many primes of Hodge-Witt reduction if and only if $B$ has infinitely many primes of ordinary reduction. Repeating this with all isotypic factors of $A$, one reduces to the case where $A$ is a product $A=A_1\times A_2\times \cdots \times A_s$ of pairwise non-isomorphic simple CM factors with CM by some algebra $P=K_1\times\cdots\times K_s$. By enlarging the field of definition $L$ of $A$, one can assume that $K_j\subseteq L$ for each $j$.  By  \Cref{le:density-in-prods} one there exists a positive density of rational primes $p$ which split almost completely in $P$. Let $p$ be  such a prime $p$. If $\wp|p$ is a prime of $L$ at which each $A$ has good reduction, then by \Cref{thm:cm-hodge-witt},  all but one of the factors $A_j$ have ordinary reduction at $\wp$ and the remaining factor has  Hodge-Witt reduction at $\wp$. From \cite[III, Prop 2.1(ii) and Prop 7.2(ii)]{ekedahl85} we see that $A$ has Hodge-Witt reduction at some prime $\wp$ of $L$ if and only if all but one factor (say $A_j$) of $A$ have ordinary reduction at $\wp$ and the remaining factor $A_j$ has Hodge-Witt reduction at $\wp$. Thus, $A$ has Hodge-Witt reduction at $\wp$. Hence, one has established that there exists a set of rational primes $p$ of positive density at which $A$ has ordinary or Hodge-Witt reduction. \Cref{elliptic-curve-example2}, \Cref{ex:ab-fourfold3} shows that the two sets of primes do not coincide in general. This completes the proof of \Cref{th:CM-ab-general}. 
\ep

\brem 
\Cref{ex:hodge-witt-cm-galois-example, ex:hodge-witt-cm-example, ex:ab-fourfold, ex:ab-fourfold2,ex:ab-fourfold3} show that the combinatorics of primes of Hodge-Witt and ordinary reduction can be quite complicated.
\erem

Let us record some obvious corollaries of \Cref{th:CM-ab-general} and \Cref{thm:cm-hodge-witt}.

\bcor\label{cor:serre-CM}
Let $L$ be a number field and let $A/L$ be an abelian variety over $L$ with complex multiplication. Then  Conjecture~\ref{serre} (of Serre) is true for $A$ i.e. every  abelian variety with complex multiplication admits infinitely many primes of good, ordinary reduction.
\ecor
\bp
This is immediate from \Cref{th:CM-ab-general}.
\ep

\bcor\label{cor:answer-hodge-witt-torsion}
Let $L$ be a number field and let $A/L$ be an abelian variety over $L$ with complex multiplication. Then
\Cref{con:positive-density-hodge-witt} (of Joshi-Rajan) is true for CM abelian varieties i.e. every abelian variety with complex multiplication admits infinitely many primes of good, Hodge-Witt reduction. 
\ecor

\bp
This is immediate from \Cref{th:CM-ab-general} and \Cref{thm:cm-hodge-witt}.
\ep

\brem 
As the Referee remarked the existence of infinitely many primes of ordinary reduction for a CM abelian variety can as also be proved via the Shimura-Taniyama formula \cite[2.1.4.1]{conrad-book} (I thank the referee for suggesting that this remark be included here). Let $P=K_1\times K_2\times\cdots \times K_s$ be a CM algebra. Let $(A/L, P\into \End_L(A)\tensor_{\Z}\Q)$ be a CM abelian variety with a CM-type $(P,\Phi)$ (see \cite[Definition 1.5.2.1]{conrad-book}). Let $\wp|p$ be a prime of $L$ lying over a rational prime $p$. Fix an isomorphism $\iota:\C\simeq \bQ_p$.  Then one has a decomposition $\Phi=\coprod_{\wq|p} \Phi_\wq$ (disjoint union over the set of primes $\wq|p$), and where $\Phi_\wq=\{ \tau\in \Phi : \iota\circ\tau \textrm{ factors through }P_\wq\}$. Then Shimura-Taniyama asserts that the Newton polygon of $A_\wp$ is concatenation of straight line segments $\Psi_\wq$ of slope $\frac{|\Phi_\wq|}{[P_\wq:\Q_p]}$ and height $|\Phi_\wq|$. Thus, to prove that $A/L$ has infinitely many primes of ordinary reduction it is enough to show that there exists infinitely many primes $p$ which split completely in $P$. This was established in the proof of \Cref{le:density-in-prods}.
Hence by the Shimura-Taniyama formula, $A$ has ordinary reduction at primes of $L$ lying above $p$. As \Cref{ex:ab-fourfold, ex:ab-fourfold2} show, the existence of primes of Hodge-Witt but non-ordinary reduction is quite subtle. It is precisely because of this that we have retained our approach via \Cref{p:almost-completely-split}, which also gives us a way of estimating densities in Hodge-Witt case and also tells when one can expect the two densities to be different.
\erem

\section{On the \mustata-Srinivas conjectures on multiplier and test ideals}\label{se:ms}
We begin with establishing the ordinarity of certain wonderful compactifications, the conjectures of \mustata-Srinivas \cite{mustata12a}, \cite{mustata12b} will be taken up in the next subsection.
\subsection{Ordinary reductions  of wonderful compactifications}
\newcommand{\sG}{\mathcal{G}}
We begin with the following general result \Cref {th:ord-wonderful-compact}, whose formulation is geared towards my discussion of the  \mustata-Srinivas conjecture in \Cref{ss:mustata-srinivas-conj}. \Cref {th:ord-wonderful-compact} leads directly, via the main characteristic $p>0$ results of  \cite{joshi11}, to the main result, \Cref{th:ord-config-spaces},  on ordinary reductions of a number of configuration spaces (the precise list of spaces is given in the statement of the theorem).
\bthm\label{th:ord-wonderful-compact}
Let $(Y,\sG)$ be a pair consisting of a geometrically irreducible, smooth, projective variety $Y/L$ over some number field $L$ and $\sG$ be a building set in the sense of \cite[Definition 2.1]{LiLi2009}, where each $G\in\sG$   is a geometrically irreducible, smooth subvariety of $Y$ is defined over $L$. For $G\in\sG$, let $I_G\subset \O_Y$ be the ideal sheaf of $G$ in $Y$ and let 
\be\label{eq:ideal-building-set} I_{\sG}=\prod_{G\in\sG}I_G\subset \O_Y\ee
be the  sheaf of ideals in $Y$ given using $\sG$.
Let $Y_{\sG}$ be the wonderful compactification (\cite[Definition 1.1]{LiLi2009}) of  $$Y^\circ=Y-\bigcup_{G\in\sG} G$$  constructed in \cite[Theorem 1.3]{LiLi2009} using $(Y,\sG)$, and let $\pi:Y_{\sG}\to Y$ be the morphism constructed by \cite[Theorem 1.3]{LiLi2009} (important point is that $\pi$ is a series of successive blowups with centered in $G\in\sG$). Assume that there exists models for $(Y,\sG)$ (in the sense of \Cref{ss:models}) and an infinite set $S$ of primes $\wp$ of $\O_L$ such that,  for all $\wp\in S$,  $Y$ and each $G\in\sG$ have good, ordinary reductions at $\wp$. Then the following assertions hold:
\benumlab
\item The canonical morphism $\pi:Y_{\sG}\to Y$ is a log resolution of $(Y,I_{\sG})$ and this property holds for the reduction modulo all but finitely many primes $\wp$ of $L$.
\item For all but  finitely  many  $\wp \in S$, the wonderful compactification $Y_{\sG}$, and each irreducible component of the exceptional divisor of $Y_{\sG}\to Y$, have good ordinary reduction at $\wp$.
\eenum
\ethm
\bp 
We begin by noting that $Y_{\sG}\to Y$ is obtained as a sequence of blowups with centers in $G\in\sG$. That $Y_{\sG}\to Y$  is a log-resolution of $(Y,I_{\sG})$ over $L$ is now immediate from \cite[Theorem 1.3]{LiLi2009}, the properties of $Y_{\sG}$ established therein, and the  definition of  $\log$-resolution of an ideal (see \cite[Definition 9.1.12]{lazarsfeld-positivity-vol2}) because the inverse image of the support, $\bigcup_{G\in\sG} G$, of $I_{\sG}$ is the exceptional divisor of $Y_{\sG}\to Y$ and by \cite[Theorem 1.2]{LiLi2009}, this exceptional divisor has simple normal crossing support.  By \Cref{ss:models}, one can find a model for $(Y,\sG)$ and $Y_{\sG}$, such that for all but finitely many prime ideals $\wp$ of $\Spec(\O_L)$, the reduction $Y_{\sG,\wp}$ of $Y_{\sG}$ at $\wp$ is obtained from the reduction mod $\wp$ datum $(Y_{\wp},\{G_{\wp}: G\in\sG\})$ and $\{G_{\wp}: G\in\sG\}$  is a building set in $Y_{\wp}$ in the sense of \cite[Definition 2.1]{LiLi2009} (we note that this building set requirement may also require the removal of a finite number of primes from $S$ as well). Then by my hypothesis that $Y$ and $\sG$ have good ordinary reduction at all but finitely many primes $\wp\in S$, and \cite[Theorem 3.2]{joshi11}, shows that $Y_{\sG}$ has good ordinary reduction at all $\wp\in S$ and $Y_{\sG,\wp}$ is obtained by a finite sequence of blowups centered in the ordinary building set $\{G_\wp:G\in \sG\}$. Notably, for all but finite number of primes $\wp$, the reduction modulo $\wp$ of the canonical morphism $Y_{\sG}\to Y$, provides a log-resolution of the reduction datum.  This proves {\bf(1,\,2)}. This completes the proof.
\ep
\newcommand{\sM}{\mathcal{M}}

\Cref{th:ord-wonderful-compact}, \cite[Theorem 3.2, Corollary 3.3]{joshi11}, the ordinarity results of the preceding subsections, especially \Cref{fanotheorem}, \Cref{th:positive-density}, and \Cref{th:CM-ab-general}, can be applied to establish \Cref{serre} (i.e. existence of ordinary reductions) for a number of configuration spaces which have appeared in literature. The precise assertion is as follows:
\bthm\label{th:ord-config-spaces}
Let $X/L$ be a smooth, proper variety over a number field $L$. If $X$ has good, ordinary reductions at infinitely many primes of $L$. Then  the following configuration spaces over $L$ arising from $X/L$ also have infinitely many primes of good ordinary reduction:
\benumlab
 		\item the scheme $X[n]$ of Fulton-MacPherson (see \cite{Fulton1994}) 
 		\item the scheme $X\langle n\rangle$ of Ulyanov (see \cite{Ulyanov2002})
 		\item the scheme $X^\Gamma$ of Kuiperberg-Thurston (see \cite{LiLi2009})
 		\item the generalized Fulton-Macpherson configuration scheme $X^{[n]}_D,X_D[n]$
 		(we assume $D$ is a smooth, ordinary subscheme of $X$) of
 		\cite{Kim2008},
 		\item the moduli, $\bar{\sM}_{0,n}$ (for $n\geq 3$), of
 		$n$-pointed stable curves of genus zero is ordinary.
 		\item the scheme of $T^{d,n}$ of stable, $n$-pointed, rooted trees of
 		$d$-dimensional projective spaces of \cite{Chen2006}.
 	\item 
 	In particular, the assertions {\bf(1)--(4)} hold whenever $X/L$ is one of the following:
 	\begin{enumerate}[label={\bf(\alph*)}]
 		\item  $X=\P^n$, or
 		\item $X= Gr(k,n)$ a Grassmanian, or
 		\item $X$ is a Fano surface, K3 surface or an abelian surface, or
 		\item $X=F_{n,m}\subset \P^{n+1}$ is a Fermat hypersurface of dimension $n\geq 1$, and degree $m\geq 2$.
 		\item $X$ is an abelian variety of any dimension with complex multiplication.
 	\end{enumerate}
 \eenum
\ethm
\bp 
The proof of {\bf(1)--(6)} is immediate from \Cref {th:ord-wonderful-compact} and \cite[Corollary 3.3]{joshi11}. The assertions {\bf(7)(a)--(7)(e)} follow if one verifies that in each case $X/L$ has good ordinary reduction at infinitely many primes of $L$. To see this, note that the assertions  {\bf(7)(a), (7)(b)} are standard: for $X=\P^n$ see \cite[Proposition 1.4]{illusie90a}, and for $X=Gr(k,n)$ one may obtain ordinarity of the reduction modulo $\wp$ by verifying that the Hodge and the Newton polygons of $Gr(k,n)$ coincide (\cite[Section 3]{illusie90a}). This is well-known and goes back to the computation of the Zeta function of Grassmannians \cite{weil49}. The assertion {\bf(7)(c)} follows from \Cref{fanotheorem}, \Cref{th:positive-density}(1) (K3 case) (or \cite{bogomolov09}) and \cite[Corollary 2.9]{ogus82} (abelian surface case). The assertion {\bf(7)(d)} follows from \Cref{tab:fermat-reductions} and \cite{toki96}. The assertion {\bf(7)(e)} follows from \Cref{th:CM-ab-general}.
\ep

\brem 
For $m> n+2$, $F_{n,m}$ is a variety of general type, so one expects that if $m\gg n+2$, then so are the configuration spaces arising from $F_{n.m}$  considered in \Cref{th:ord-wonderful-compact}{\bf(1)--(4)}. 
\erem

\subsection{Two remarks on the  \mustata-Srinivas Conjecture}\label{ss:mustata-srinivas-conj}
\newcommand{\sX}{\mathcal{X}}
\newcommand{\sY}{\mathcal{Y}}
\newcommand{\fa}{\mathfrak{a}}
\newcommand{\sI}{\mathcal{I}}
\newcommand{\sD}{\mathcal{D}}
\newcommand{\ssI}{\mathscr{I}}
In  this subsection, we discuss a conjecture considered  in\cite[Conjecture 1.2]{mustata12a} and its relationship to \Cref{con:positive-density-hodge-witt}. Notably, we apply \Cref{th:ord-wonderful-compact}, \Cref{th:ord-config-spaces} to establish new cases of \cite[Conjecture 1.2]{mustata12a}.

We begin by setting up the notations for \cite{mustata12a,mustata12b}.
Let $Y$ be a smooth variety over a number field $L$, $I\subset \O_Y$ be an ideal sheaf on $Y$ and suppose  $D\subset Y$ is a Weil divisor on $Y$. Suppose we are given a model $\sY\to\Spec(\O_L)$ for $Y$ and models $\sD\subset \sY$ (resp. $\sI\subset \O_{\sY}$) for $D$ (resp. $I\subset\O_Y$)   (models are as in \Cref{ss:models}--but we will not require properness of $Y$). For $\wp\in \Spec(\O_L)$ let $\sY_\wp$  be the fiber of $\sY$ over $\wp$ and let $\sI_\wp$ (resp. $\sD_\wp$) be the restriction of $\sI$ (resp. $\sD$) to the fiber $\sY_\wp$. Let $\lambda\in\R_{\geq0}$ be a non-negative real number. Let $\ssI(\sY,\sD,\sI^\lambda)$ be the multiplier ideal of $(\sY,\sD,\sI^\lambda)$ (see \cite[Section 3, ]{mustata12a}), and let $\tau(\sY_\wp,\sD_\wp,\sI_\wp^\lambda)$ be the test ideal of $(\sY_\wp,\sD_\wp,\sI_\wp^\lambda)$ (see \cite[Section 3, ]{mustata12a}).  I will not state \cite[Conjecture 1.1]{mustata12a} as it is implied by \Cref{serre}, on the other hand, the second conjecture of  \mustata-Srinivas \cite[Conjecture 1.2]{mustata12a} asserts the following:
\bcon\label{con:ms}
With the above assumptions and conventions on $(Y,D,I)$ and models $(\sY,\sD,\sI)$ (as in \Cref{ss:models}), 
the equality $$\ssI(Y,D,I^\lambda)_\wp:=\ssI(\sY,\sD,\sI^\lambda)_\wp=\tau(\sY_\wp,\sD_\wp,\sI_\wp^\lambda)=:\tau(Y_\wp,D_\wp,I_\wp^\lambda)$$
holds  for a  dense subset  of prime ideals $\wp\in\Spec(\O_L)$ and for all $\lambda\in\R_{\geq0}$.
\econ
The first remark is the relationship between \Cref{con:positive-density-hodge-witt}, \Cref{serre}, \Cref{con:ms} and is as follows. 

\bthm\label{th:jr-ms}
Consider the following assertions:
\benumlab
\item \Cref{con:positive-density-hodge-witt} of Joshi-Rajan holds for all smooth, projective varieties over any number field. 
\item \Cref{serre} of Serre holds for all smooth, projective varieties over any number field.
\item Conjecture 1.1 of  \mustata-Srinivas \cite{mustata12a} holds for all smooth, projective varieties over $\bQ$, and hence holds for all algebraically closed fields of characteristic zero.
\item \Cref{con:ms} i.e.  \cite[Conjecture 1.2]{mustata12a} holds for all smooth varieties over any algebraically closed field of characteristic zero.
\eenum
Then {\bf(1)}$\iff${\bf(2)}$\implies${\bf(3)}$\iff${\bf(4)}. In particular,  {\bf(1)}$\implies${\bf(4)} i.e. \Cref{con:positive-density-hodge-witt} of the existence of primes of Hodge-Witt reduction for all smooth, projective varieties over number fields implies the \Cref{con:ms} i.e. \cite[Conjecture 1.2]{mustata12a} on the coincidence of multiplier and test ideals at infinitely many reductions.
\ethm

\bp
The implication {\bf(1)}$\iff${\bf(2)} is immediate from \Cref{th:equivalence}. The equivalence {\bf(3)}$\iff${\bf(4)} is proved in \cite{mustata12a} and \cite{mustata12b} (the only point to note is that the numbering Conjecture 1.1 and Conjecture 1.2 is switched between \cite{mustata12a} and \cite{mustata12b}). The implication {\bf(2)}$\implies${\bf(3)} is immediate from the property \Cref{eq:ord} of ordinary varieties in characteristic $p>0$. This implies that Frobenius morphism $H^i(X,\O_X)\to H^i(X,\O_X)$ is an isomorphism for all $i\geq 0$, and the case $i=\dim(X)$ is adequate for the implication {\bf(2)}$\implies${\bf(3)}. We note that by \cite[Proposition 5.3]{mustata12a}, if {\bf(3)} holds over $\bQ$, then {\bf(3)} holds over all algebraically closed fields of characteristic zero. The equivalence {\bf(3)}$\iff${\bf(4)} is \cite[Theorem 1.3]{mustata12a} and \cite[Theorem 1.3]{mustata12b}. 
\ep

 My next results \Cref{cor:mustata-srinivas-cm}, and especially \Cref{cor:ord-config-spaces}, are motivated by the ordinarity results of the present paper,  the  computations of multiplier ideals for hyperplane arrangements  \cite{mustata2006}, \cite{teitler2007}--especially the latter which proceeds via wonderful compactifications of \cite{LiLi2009}, and by the ordinarity of  wonderful compactifications and configuration spaces  arising from ordinary complete varieties established in \cite{joshi11}. This allows us to prove \Cref{con:ms} for wonderful compactifications in \Cref{th:ord-config-spaces} and a large class of configuration spaces in \Cref{th:ord-config-spaces} and \Cref{cor:ord-config-spaces}.

\newcommand{\sC}{\mathcal{C}}
\bthm\label{cor:mustata-srinivas-cm}
Let all the assumptions and the notations ($L, (Y,\sG)$, $\pi:Y_{\sG}\to Y$, $I_{\sG}\subset \O_Y$), of \Cref{th:ord-wonderful-compact} remain in force. Notably, $\sG$ is a building set in $Y$ and $(Y,\sG)$ have good ordinary reduction modulo an infinite set of primes $\wp\in S$. Then there exists an infinite set $S$ of primes of $L$ such that \Cref{con:ms} of  \mustata-Srinivas holds for the pair $(Y,I_{\sG})$ i.e.  for all primes $\wp\in S$ and all $\lambda\geq0$ one has
$$\ssI(Y,I_{\sG}^\lambda)_\wp=\tau(Y_\wp,I^\lambda_{\sG,\wp}).$$ 
\ethm
\bp 
We equip ourselves with models (\Cref{ss:models}) for all the relevant data including $Y,Y_{\sG},\pi$, but for notational sanity we suppress models from the notation altogether.  We assert that the method of proof of \Cref{th:jr-ms}{\bf(3)}$\implies${\bf(4)} given in \cite{mustata12a}, together with \Cref{th:ord-wonderful-compact} implies the claim. 

Indeed, by definition of building sets \cite[Definition 2.1]{LiLi2009}, intersections of members of $\sG$ are in $\sG$, and hence by \Cref{th:ord-wonderful-compact}, \cite[Theorem 1.3]{LiLi2009}, and  \cite[3.2]{joshi11},  the intersections of arbitrary, pairwise distinct irreducible components of the exceptional locus of $Y_{\sG}\to Y$ are smooth and ordinary. 

This means that one has the conclusion of \cite[Lemma 5.6]{mustata12a} available to us. We note that in that Lemma, semi-simplicity of Frobenius is asserted (and this is used to establish bijectivity  \eqref{eq:ord} of relevant Frobenius morphism \cite[2.1]{mustata12a}). On the other hand,  as observed in \cite{joshi11}, the said bijectivity \eqref{eq:ord} is available to us  by our hypothesis that  $\sG$ is an ordinary building set (so intersections of any members of $\sG$ are again in $\sG$, and that members of $\sG$ have ordinary reduction). This also gives us the conclusions of \cite[Corollaries 5.7 and 5.8]{mustata12a} and hence one deduces the validity of \cite[Theorem 5.10]{mustata12a} in the present case.

Specifically, let $E$ be the exceptional divisor of $\pi$. Consider the reduction modulo $\wp$ of the relevant data for our chosen model. Thus, $E_\wp$ is the exceptional divisor of $\pi_\wp$. Let $F=F_{\wp}$ be the Frobenius morphism at $\wp$.  Then for all integers $e\geq 1$, and  for all but finitely many $\wp\in S$, a surjection of coherent sheaves on $Y_\wp$ (given by \cite[Theorem 5.10]{mustata12a}):
\be\label{eq:surj}
(\pi_{\wp})_*(F_*^e(\omega_{Y_{\sG,\wp}}(E_\wp))) \twoheadrightarrow (\pi_{\wp})_*(\omega_{Y_{\sG,\wp}}(E_\wp)).
\ee
\Cref{con:ms} asserts (for all $\lambda\geq 0$), the equality of the test ideal $\tau(Y_\wp,I_{\sG_\wp}^\lambda)$ at $\wp$, and  the reduction modulo $\wp$, $\ssI(Y,I_\sG^\lambda)_\wp$, of the multiplier ideal.   Note that inclusion $\tau(Y_\wp,I_{\sG_\wp}^\lambda)\subseteq \ssI(Y,I_\sG^\lambda)_\wp$ in one direction is given to us by invoking  \cite[Proposition 4.2]{mustata12a} (with $\pi_{\wp}$ as above and $D_\wp=0$). Thus,  the test ideal at $\wp$ is contained  in the multiplier ideal obtained by reducing modulo $\wp$. Hence, one has to prove the reverse inclusion (this is also how the proof of \cite[Theorem 1.3]{mustata12a} proceeds).

This remaining inclusion assertion follows from \cite[Lemma 6.1]{mustata12a}. This is seen as follows, the surjection \eqref{eq:surj} is local on $Y_\wp$. So we may cover $Y$ by some affine open cover and for opens  $U\subset Y$, replace $\pi:Y_{\sG}\to Y$, by $\pi^{-1}(U)\to U$ (for $U$ in the cover),  and choose a function $h_\wp\in\Gamma(U_\wp,\O_{U_\wp})$, such that $div(\pi_\wp^*(h_\wp))$ has support in the exceptional divisor of $\pi_\wp$ over $U_{\wp}$, and such that all the remaining hypothesis of \cite[Lemma 6.1]{mustata12a} are satisfied for all members of the open cover of $Y_\wp$  (the crucial hypothesis of that Lemma being the availability of the above surjection \eqref{eq:surj} on $U_\wp$ for $\pi_{\wp}^{-1}(U_\wp)\to U_{\wp}$ for  $U_\wp$ in the open cover of $Y_\wp$). This, allows us to invoke \cite[Lemma 6.1]{mustata12a} to assert that, for $U_\wp$ in the cover and hence for $Y_\wp$, the reduction of the multiplier ideal at $\wp$ is contained in the test ideal at $\wp$. Thus, \Cref{con:ms} holds for $(Y,I_{\sG})$. This completes the proof.
\ep

\bcor\label{cor:ord-config-spaces}
Suppose $X/L$ is as in \Cref{th:ord-config-spaces}{\bf(7)(a)--(7)(e)}. If $(\sC,I)$ is one of the configuration space given in \Cref{th:ord-config-spaces}{\bf(1)--(6)} arising from $X/L$,  and $I=I_{\sG}$ is the ideal sheaf defined by \eqref{eq:ideal-building-set}  for the building set $\sG$ used in \cite{LiLi2009}  for constructing $\sC$, then \Cref{con:ms} holds for $(\sC,I)$. 
\ecor
\bp 
This is immediate from \Cref{cor:mustata-srinivas-cm}, using the building sets used in \cite{LiLi2009} for constructing the configurations spaces \Cref{th:ord-config-spaces}{\bf(1)--(6)} arising from $X/L$.
\ep

\section{Primes of non Hodge-Witt reduction}\label{se:non-hw}
Let us discuss non Hodge-Witt reduction. A smooth,projective variety over an algebraically closed field $k$ of characteristic $p>0$ is a \emph{non Hodge-Witt variety} if, for some $i,j\geq$, the de Rham-Witt cohomology group $H^i(X,W\Omega_X^j)$ is not finitely generated as a $W$-module. In other words, for some $i,j\geq 0$, $H^i(X,W\Omega^j_X)$ has infinite torsion.	

Now suppose $X/K$ is a smooth, projective variety defined over a number field. Choose a model $\mathcal{X}\to\Spec(\O_K)$ (\Cref{ss:models}). I will say that $X$ has non Hodge-Witt reduction at a prime  $\wp$ of good reduction if $X_\wp$ is a non Hodge-Witt variety (viewed as a variety over the algebraic closure of the residue field of $\wp$). 

In this section we briefly describe what we know about primes of non Hodge-Witt reduction and what we expect to be true and how this relates to other well-known conjectures. A surprising and counter-intuitive observation is \Cref{th:frequency-of-non-hodge-witt-cm} which asserts that a Fermat hypersurface $F_{n,m}\subset \P^{n+1}$ of dimension $n$  and degree $m\geq 211$, at least 98\% of primes are of non-Hodge-Witt reduction. Readers will also find other examples in \Cref{se:examples} useful.

\subsection{Hodge-Witt torsion}
We include here some  observations
probably well-known to the experts, but we have not
found them in print.  We assume as in
the previous section that $X/K$ is smooth projective variety over a
number field and that we have fixed a proper flat model, smooth
over some open subscheme of the ring of integers of $K$ and whose
generic fiber is $X$.

Before we proceed, we record the following:

\begin{proposition}[Joshi-Rajan] Let $X/K$ be a smooth projective variety over a
number field $K$.  Then there exists an integer $N$ such that for all
primes $\wp$ in $K$ lying over any rational prime $p\geq N$, the
following dichotomy holds
\begin{enumerate}
\item either for all
$i,j\geq 0$, the Hodge-Witt groups are free $W$-modules (of finite
type), or
\item there is some pair $i,j$ such that
$H^i(X_\wp,W\Omega_{X_\wp})$ has infinite torsion.
\end{enumerate}
\end{proposition}

\begin{proof}
   Choose a finite set of primes $S$ of $K$, such that $X$ has a
proper, flat model  over $\Spec(\O_{K,S})$, where $\O_{K,S}$
denotes the ring of $S$-integers in $K$. Choose $N$ large enough
so that for any prime $\wp$ lying over a rational prime $p>N $,
we have $\wp \not\in S $, and all  the crystalline cohomology
groups of $X_\wp$ are torsion-free. We note that this choice of
$N$ may depend on the choice of the model for $X$
over $\Spec(\O_{K,S})$. If $\wp$ is such that
$H^i(X_\wp,W\Omega^j_{X_\wp})$ are all finite type,  then by the
degeneration of the slope spectral sequence at the $E_1$-stage by
\cite[Theorem~3.7]{illusie79b}, and the
fact that the crystalline cohomology groups are torsion free, it
follows that the Hodge-Witt groups are free as well. If, on the
other hand,  some Hodge-Witt group of $X_\wp$ is not of finite
type over $W$, then we are in the second case.  \end{proof}

\begin{question}[Joshi-Rajan]\label{hodge-witt-torsion}
   Let $X/K$ be a smooth projective variety over a number
field.  When does there exist an infinite set of primes of
$K$ such that the Hodge-Witt cohomology groups of the reduction
$X_\wp$ at $\wp$ are not Hodge-Witt?
\end{question}

We would like to explicate the information encoded in such a set of primes (when it exists).

\begin{proposition}[Joshi-Rajan]\hfill
\benumlab
\item Let $X/K$ be an abelian surface over a number field $K$. Then there
exists infinitely many primes $\wp$ such that
$H^2(X_\wp,W(\O_{X_\wp}))$ has infinite torsion if and only if there
exists infinitely many primes $\wp$ of supersingular reduction for
$X$. 
\item In particular, let $X=E\times_KE$, with $E$ an elliptic curve over a number field  $K$ with  a real embedding. Then for an infinite set of primes of $K$, the
Hodge-Witt groups $H^i(X_\wp,W\Omega_{X_\wp}^j)$ are not torsion free for
$(i,j)\in\left\{(2,0),(2,1)\right\}$.
\eenum
\end{proposition}

\begin{proof}
The first part  follows from  the results of
\cite[Section~7.1(a)]{illusie79b}.  The second part follows from
combining the first part  with
Elkies' theorem (see \cite{elkies87,elkies89}),  that given an elliptic curve
$E$ over a number field $K$ with a real embedding,  there are infinitely
many primes $\wp$ of $K$ such that $E$ has supersingular reduction modulo $\wp$.
\end{proof}

\subsection{Existence of primes of non Hodge-Witt reduction}
\br
Question~\ref{hodge-witt-torsion} which was raised by us in \cite{joshi00a-arxiv-version} can now be formulated as a more precise conjecture (in the light of Corollary~\ref{cor:answer-hodge-witt-torsion}). Our formulation is the following.
\er

\bcon\label{con:non-hodge-witt}
Let $X/K$ be any smooth, projective variety of dimension $n$ over a number field $K$, then
\benumlab
\item  there exists infinitely many primes $\wp$ of $K$ such that $X$ has non-Hodge-Witt reduction at $\wp$.
\item and if  $H^n(X,\O_X)\neq 0$, then there exists infinitely many primes $\wp$ of $K$ such that the domino associated to the differential $$H^n(X_\wp,W(\O_{X_\wp}))\to H^n(X_\wp,W\Omega^1_{X_\wp})$$ is non-zero (here $X_\wp$ is the ``reduction'' of $X$ at $\wp$). In particular, $X$ has non-Hodge-Witt reduction at $\wp$.
\eenum
\econ

\br\label{r:supersingular-k3}\ 
\benumlab
\item The difference between \Cref{con:non-hodge-witt}{\bf(1)} and \Cref{con:non-hodge-witt}{\bf(2)} is that the former asserts that the differential $H^j(X_\wp,W\Omega_{X_\wp}^i) \to H^j(X_\wp,W\Omega_{X_\wp}^{i+1})$ is non-zero for some $i,j\geq0$, while the latter specifies that the specific differential $H^n(X_\wp,W(\O_{X_\wp}))\to H^n(X_\wp,W\Omega^1_{X_\wp})$ is non-zero. In particular, 

\ \ \ \Cref{con:non-hodge-witt}{\bf(2)}$\implies$\Cref{con:non-hodge-witt}{\bf(1)}.

\item  The interest study of the top differential  in \Cref{con:non-hodge-witt}{\bf(2)}  is natural because (a) it contributes to the middle dimensional crystalline cohomology $H^n_{cris}(X/W)$ and (b) by Lefschetz type theorems, lower dimensional cohomology arises from hypersurface sections, and hence this differential can be viewed as highlighting new phenomena in dimension $n$  (with phenomena from lower dimensions as being considered ``understood.''). 
\item Moreover, by \cite{nygaard79b}, for $\dim(X)=2$, the only non-zero differential in the slope spectral sequence of a smooth, proper surface is the differential $H^2(X,W(\O_X))\to H^2(X,W\Omega^1_X)$.

\item  There is also an analog of the remark {\bf(3)} for some threefolds. By the criterion for degeneration of slope spectral sequences of threefolds given by \cite[Theorem 6.1]{joshi07} and the vanishing of $H^1(X,W(\O_X)=H^2(X,W(\O_X)))=0$ given by \cite{suwa93}, one sees that if $X$ is any smooth, projective, complete intersection threefold, then $X$ is Hodge-Witt if and only if the differential $H^3(X,W(\O_X))\to H^3(X,W\Omega_X^1)$ is zero (equivalently $H^3(X,W(\O_X)))$ is a finite type $W$-module. Thus for such $X$, this top differential controls the de Rham-Witt cohomology.

\item To understand the meaning of \Cref{con:non-hodge-witt}{\bf(2)}  in a concrete situation, let  $X$ be a K3 surface. Then  $H^2(X,\O_X)\neq 0$, hence the hypothesis of \Cref{con:non-hodge-witt}{\bf(2)} is valid. Now, by  \cite[II, 7.2(b)]{illusie79b}, $X$ has non Hodge-Witt reduction at a prime $\wp$ of $K$ if and only if $X_\wp$ is a supersingular K3 surface if and only if $X_\wp$ is of not of finite height. In this (supersingular) case, one has $H^2(X_\wp,W(\O_{X_\wp}))=k_\wp[[V]]$ where $k_\wp$ is the residue field of $\wp$ (see \cite[II, 7.2(b)]{illusie79b}). Thus \Cref{con:non-hodge-witt}{\bf(2)} predicts that there exist infinitely many primes of supersingular reduction for a K3 surface over a number field.
\item Now suppose $X$ is smooth, projective and Calabi-Yau threefold i.e. $\dim(X)=3$. Then \Cref{con:non-hodge-witt}{\bf(2)} and remark{\bf(4)} imply that $X/K$ has infinitely many reductions  $X_\wp$ which are not of finite height in the sense that the Artin-Mazur formal group of $X_\wp$ is not of finite height (see \cite{vandergeer2003}).
\eenum
\er

\subsection{Non Hodge-Witt reduction and Elkies' Theorem on supersingular primes}\label{r:elkies} 
\Cref{con:non-hodge-witt} includes, as a special case,  the infinitude of primes of supersingular reduction for any elliptic curve over any number field $K$ (\Cref{con:non-hodge-witt-implies-elkies-thm}{\bf(1)}) and the existence of infinitude of primes of common primes of supersingular reduction for any pair of elliptic curves over a number field \Cref{con:non-hodge-witt-implies-elkies-thm}{\bf(2)}.

\bpro\label{con:non-hodge-witt-implies-elkies-thm}
Assume \Cref{con:non-hodge-witt} is true for all abelian surfaces over  number fields. Then 
\benumlab 
\item every  elliptic curve over any number field  has infinitely many primes of supersingular reduction, and
\item every pair of elliptic curves $E_1,E_2$ over any number field $K$ have infinitely many common primes of supersingular reduction.
\eenum
\epro
\bp 

To see {\bf(1)} this let $E/K$ be an elliptic curve over any number field $K$.
Let $A=E\times_KE$. Then $A$ is an abelian surface and hence $H^2(A,\O_A)\neq 0$ and so \Cref{con:non-hodge-witt} predicts that there are infinitely many primes $\wp$ of $K$ such that $A$ has non-Hodge-Witt reduction at $\wp$. But since $\dim(A)=2$, $A$ has non-Hodge-Witt reduction if and only if the $p$-rank of $A$ is zero (\Cref{thm:cm-hodge-witt-3}). This means $E$ has $p$-rank zero at all such $\wp$. Thus, $E$ has supersingular reduction at infinitely many primes $\wp$ of $K$.

To prove \Cref{con:non-hodge-witt-implies-elkies-thm}{\bf(2)}, let $E_1,E_2$ be a pair of elliptic curves over a number field $K$, and let $A=E_1\times_K E_2$. \Cref{con:non-hodge-witt} predicts that $A$ has non-Hodge-Witt reduction modulo $\wp$ for infinitely many $\wp$. Now if $E_1$ or $E_2$ is ordinary modulo $\wp$ then $A$ has Hodge-Witt reduction modulo $\wp$. Thus, $A$ is not Hodge-Witt at $\wp$ implies both $E_1,E_2$ are not ordinary modulo $\wp$ i.e. both are supersingular modulo $\wp$.
\ep

\br \label{r:supersingular-two-elliptic} The set of such primes given by \Cref{con:non-hodge-witt-implies-elkies-thm}{\bf(2)} is, of course, expected to be very thin and in the next remark we discuss an explicit example.

Let us note that even for pairs of elliptic curves over $\Q$ with small conductors, the first prime of common supersingular reduction can be very large. Here is a particularly interesting example with ``modular flavor.'' Consider the modular curve $X_0(37)$, which has genus two and its Jacobian, $J_0(37)$, is isogenous to a product of two elliptic curves of conductor $N=37$. The curves were described by Mazur and Swinnerton-Dyer in their classic paper on Weil curves and the two elliptic curve factors of $J_0(37)$ are not isogenous to each other. The two curves appear as 37a and 37b1 in John Cremona's tables. So to find a prime of non-Hodge-Witt reduction for $J_0(37)$ is the same as finding a prime of common supersingular reduction for the elliptic curve factors of $J_0(37)$. The first such a prime is $p=18489743$ (a number several order of magnitudes larger than the conductor). This makes searching for such primes for pairs of elliptic curves rather tedious.
\er

\subsection{Conjecture~\ref{con:non-hodge-witt} for abelian varieties and Fermat varieties}
\numberwithin{table}{subsection}
For abelian varieties \Cref{con:non-hodge-witt} asserts the following:
\bpro\label{con:non-hodge-witt-implies-elkies-thm3}\label{pr:ab-var-non-hw}
Let $A$ be an abelian variety over a number field $K$ with $\dim(A)=g>1$. If \Cref{con:non-hodge-witt} is true for $A$, then there exist infinitely many primes  $\wp$ of $K$ such that reduction of $A$  at $\wp|p$ has $p$-rank at most $g-2$.
\epro
\bp 
This is immediate from \Cref{thm:cm-hodge-witt-3}.
\ep

Next let us consider Fermat hypersurfaces.  Let $F_{m,n}$ be the closed subscheme of $\P^{n+1}$ defined by the homogeneous equation
\be\label{eq:fermat-hyp}
	 F_{n,m}: X_0^m+\cdots +X_{n+1}^m=0.
\ee  
The notation in \eqref{eq:fermat-hyp} is such that $F_{n,m}$ is the \emph{Fermat hypersurface of dimension $n$ and degree $m$}.

In the direction of \Cref{con:non-hodge-witt} we have the following:
\bthm\label{th:fermat-non-hw1}
\label{thm:non-hodge-witt-CM}
Let $X$ be one of the following:
\benumlab
\item a Fermat hypersurface $F_{n,m}\subset \P^{n+1}$  over $\Q$ of dimension $n$ and degree $m$, then \Cref{con:positive-density-hodge-witt}, \Cref{serre} hold for all $(n,m)$ with $n,m\geq1$,
 and if
$$(n,m)=\begin{cases}
n=2, m\geq 4\\
n\geq3, m\geq 5,
\end{cases}
$$
then 
 \Cref{con:non-hodge-witt}{\bf(1)} holds for $F_{n,m}$.
\item If $X$ is 
\begin{enumerate}[label={\bf(\alph{*})}]
\item
   a Fermat hypersurface $F_{n,m}\subset \P^{n+1}$  over $\Q$ of dimension $n\in \{2,3\}$ and degree $m$, or
\item a simple abelian variety with complex multiplication over some number field $K$.
 \eenum
Then \Cref{con:non-hodge-witt}{\bf(2)} (and hence \Cref{con:non-hodge-witt}{\bf(1)}) holds for $X$.
\eenum
\ethm

\bp
\Cref{thm:non-hodge-witt-CM}{\bf(1)} was observed in (see \cite{joshi00a-arxiv-version}). Proof of this assertion uses the \Cref{tab:fermat-reductions} which summarizes the results of \cite{toki96} and the last column of that table tells us (under the stated hypothesis) that there are infinitely primes of non-Hodge-Witt reduction for $X$. Especially that one has, from \Cref{tab:fermat-reductions}, that for $(n,m)$ given  in the assertion \Cref{thm:non-hodge-witt-CM}{\bf(1)}, that the density $\dnhwa{n,m}=\dnhwa{F_{n,m}}$ of primes of non Hodge-Witt reduction is given by
$$\dnhwa{n,m}=\begin{cases}
1-\frac{1}{\phi(n)} & \text{ if }(n,m)\neq (2,7),  \\
\frac{1}{2} & \text{ if }(n,m)=(2,7).
\end{cases}
$$ 
This completes the proof of \Cref{thm:non-hodge-witt-CM}{\bf(1)}. 

The assertion of Theorem~\ref{thm:non-hodge-witt-CM}{\bf(2)(a)} follows from \cite[Theorem 6.1]{joshi07} and  \cite[Theorem 1]{suwa93} as $H^2(X_\wp,W(\O_{X_\wp}))=0$ (resp. $H^1(X_\wp,W(\O_{X_\wp}))=0$) for $n=3$ (resp. $n=2$). So if $F_{n,m}$ is non-Hodge-Witt for some $\wp$ and $n=2$ or $n=3$, then $H^n(X_\wp,W(\O_{X_\wp}))$ has infinite torsion as asserted.

The \Cref{thm:non-hodge-witt-CM}{\bf(2)(b)} is rather more delicate. After \cite[Theorem 1.1]{sugiyama14}, the existence primes of non Hodge-Witt reduction boils down to a Chebotarev argument similar to the one used in the proof of Theorem~\ref{thm:cm-hodge-witt-2}. To be precise one needs to show that there are infinitely many primes $p$ such that at least two of the  primes $\wp_1,\ldots,\wp_r$ lying over $p$ in $K_0$ remains inert in $K$ (which is clear by the method of proof of Theorem~\ref{thm:cm-hodge-witt-2}). So the existence of primes $\wp$ of non Hodge-Witt reduction is clear. So suppose that for some prime $\wp|p$, $X_\wp$ is non-Hodge-Wit. Let $g=\dim(X_{\wp})$ be the dimension of $X_\wp$.

Suppose, if possible, that \Cref{con:non-hodge-witt}{\bf(2)} fails to be true i.e. $H^g(X_\wp,W(\O_{X_\wp}))$ is of finite type over $W$. Then I claim that this leads to a contradiction. If $H^g(X_\wp,W(\O_{X_\wp}))$ is of finite type then, by \cite[Chapitre I, 2.18.2]{Illusie1983a}, the domino number $T^{0,g}=0$. 

Since $X_\wp$ is an abelian variety, its crystalline cohomology is torsion-free and Hodge-de Rham spectral sequence of $X_\wp$ degenerates at $E_1$ (\cite{illusie79b}) i.e. according to \cite[2.31]{joshi2020}, $X_\wp$ is Mazur-Ogus variety. This means  that the Hodge-Witt numbers $h_W^{i,j}$ \cite[2.24]{joshi2020} (defined by Torsten Ekedahl) coincides with Hodge numbers $h^{i,j}$ of $X_\wp$ (\cite[2.25.2]{joshi2020}). For a slope   $\lambda$ of Frobenius $F_\wp$ occurring in $H^g_{cris}(X_\wp/W(k_\wp))$, let $m_\lambda$ be its multiplicity.  Let $$m^{0,g}= \sum_{\lambda\in[-1,0[}(\lambda-0+1)m_\lambda+\sum_{\lambda\in[0,1[}(0+1-\lambda)m_\lambda$$ be the slope-number of $X_\wp$ given \cite[2,22]{joshi2020}. Then from the proof of \cite[Theorem 2.5]{joshi2020} that
$$0=T^{0,g}=h_W^{0,g}-m^{0,g}=h^{0,g}-m^{0,g}.$$
Since $X_\wp$ is an abelian variety of dimension $g$, $$\dim H^g(X_\wp,\O_{X_\wp})=h^{0,g}=1.$$
Hence, $T^{0,g}=0$ says that $m^{0,g}=1$. From the standard properties of crystalline cohomology of an abelian variety, $\lambda\geq0$. Hence the first term defining  $m^{0,g}$ is zero. Since $X_\wp$ has non Hodge-Witt reduction, I have shown \cite[Theorem 2.2]{joshi2016ord} that any slope of $H^g_{cris}(X_\wp/W(k_\wp))$ satisfies $\lambda\geq 1$. Hence the second term defining $m^{0,g}$ is also zero. This says $m^{0,g}=0$. This contradicts $m^{0,g}=1$. Thus $T^{0,g}\neq 0$ (and hence $T^{0,g}=1$) and hence $H^g(X_\wp, W(\O_{X_\wp}))$ is not a finite type $W$-module. Therefore it has infinite torsion. This proves \Cref{thm:non-hodge-witt-CM}{\bf(2)(b)} and hence the theorem.
\ep

\vskip1cm
\begin{table}[H]
\begin{center}
	\renewcommand{\arraystretch}{1.2}
	\captionof{table}{\small Densities of ordinary, Hodge-Witt, and non-Hodge-Witt reductions of Fermat hypersurfaces $F_{n,m}\subset \P^{n+1}/\Q$ of dimension $n$ and degree $m$}\label{tab:fermat-reductions}
	\begin{tabular}{|c|c|c|c|}
		\hline $(n,m)$ & ordinary $\dor{n}{m}$ & Hodge-Witt $\dhw{n}{m}$ & non Hodge-Witt $\dnhw{n}{m}$ \\
		\hline $(1,m)$ & $\frac{1}{\phi(m)}$ & $1$ & $0$ \\
		\hline $(n,1)$ & $1$ & $1$ & $0$ \\
		\hline $(n,2)$ & $1$ & $1$ & $0$ \\
		\hline $(2,3)$ & $\frac{1}{2}$ & $1$ & $0$ \\
		\hline $(3,3)$ & $\frac{1}{2}$, ($p\cong 1\bmod{3}$) & $1$ & $0$  \\
		\hline $(3,4)$ & $\frac{1}{2}$, ($p\cong 1\bmod{4}$) & $1$ & $0$ \\
		\hline $(5,3)$ &  $\frac{1}{2}$, ($p\cong 1\bmod{3}$) & $1$ & $0$  \\
		\hline $(2,7)$ & $\frac{1}{6}$, ($p\cong 1\bmod{7}$) & $\frac{1}{2}$, ($p\cong 1,2,4\bmod{7}$)  & $\frac{1}{2}$, ($p\cong 3,5,6\bmod{7}$) \\
		\hline $(n,m)$ & $\frac{1}{\phi(m)}$ & $\frac{1}{\phi(m)}$ & $1-\frac{1}{\phi(m)}$ \\
		\hline
	\end{tabular}
\end{center}
\end{table}
\br
As \Cref{th:frequency-of-non-hodge-witt-cm} shows that in fact primes of non Hodge-Witt reduction are surprisingly prevalent in the context of \Cref{th:fermat-non-hw1}{\bf(1)}.
\er
\section{Explicit examples}\label{se:examples}
In this subsection we describe some explicit examples related to \Cref{cor:serre-CM}, \Cref{con:positive-density-hodge-witt} and \Cref{con:non-hodge-witt} which the reader may find useful.
\subsection{Fermat Varieties}\label{ss:fermat}
Let us begin with Fermat hypersurface $F_{m,n}$ defined by  \eqref{eq:fermat-hyp}, considered as a variety over $\Q$, and explicitly compute densities of primes of various types of reduction. Let $\dor{n}{m},\dhw{m}{n}$, and $\dnhw{n}{m}$ denote the asymptotic density of primes of ordinary, Hodge-Witt and non-Hodge-Witt reduction for $F_{n,m}/\Q$ (we ignore primes of bad reduction). Then the densities may be calculated by using the results of \cite{toki96}. Except for a finite list of exceptional $F_{n,m}$, which are given in \Cref{tab:fermat-reductions}, the densities are given by the following (we write $\delta_{n,m}^?$ for $\delta^?(F_{n,m},\Q)$, with $?={\rm ord,hw,non-hw}$):
\begin{eqnarray}
  \delta_{n,m}^{\rm ord} & = & \frac{1}{\phi(m)}\\
  \delta_{n,m}^{\rm hw} & = & \frac{1}{\phi(m)}\\
  \delta_{n,m}^{\rm non-hw} & = & 1-\frac{1}{\phi(m)}
\end{eqnarray}
In the table we list the density of primes of different types of reduction for $F_{n,m}$, we also list the condition for a prime so that $F_{n,m}$ has the reduction of the given sort.
The table also shows that the density of ordinary and Hodge-Witt reductions can be different. In particular, for the septic Fermat surface $F_{2,7}\subset\P^3$ the densities are \be\dor{2}{7}=\frac{1}{6}<\dhw{2}{7}=\frac{1}{2}=\dnhw{2}{7}.\ee So from \Cref{tab:fermat-reductions}, we see that among all Fermat surfaces $F_{2,n}$ ($n\geq 4$), the septic Fermat surface shows rather interesting behavior. We do not know why this surface shows this exceptional behavior.

\subsection{The surprising prevalence of non Hodge-Witt (and non ordinary) reduction}\label{ss:prevalence}
\bthm\label{th:frequency-of-non-hodge-witt-cm} 
Suppose $n\geq 3$ and consider the Fermat hypersurface $F_{n,m}\subset \P^{n+1}$ of degree $m$. If $m\geq 211$, then $F_{n,m}$ has non-Hodge-Witt reduction at  more than $98\%$ of primes of $\Q$. Notably,  for all $n\geq 2$, one has
$$\lim_{m\to\infty} \delta^{\rm non-hw}(F_{n,m})=1.$$
\ethm
\bp
From \cite{toki96} whose main result is summarized in the Table~\ref{tab:fermat-reductions}, one obtains the density of non-Hodge-Witt primes in terms of the Euler's totient function $\phi(m)$. One can use elementary lower bounds for the Euler totient function. There are many ways to get lower bounds for $\phi(m)$. The easiest case is to get lower bounds  arising from \cite[Theorem 327]{hardy-wright} which asserts that for any positive $\delta$, $\phi(m)/m^{1-\delta}\to\infty$ as $m\to\infty$. There are many well-known bounds which follow from this.  For example, it is a pleasant exercise to check that $\phi(m)\geq \frac{\sqrt{m}}{\sqrt{2}}$  holds for $m\geq 3$. 

Hence, if $m\geq 5202$ then $\phi(m)\geq  51$. Then by examining values of $\phi(m)$ for $1\leq m\leq5202$ one sees in fact that if $m\geq 211$ then $\phi(m)\geq 51$. So one has  $$1-\frac{1}{\phi(m)}\geq 1-\frac{1}{51}\approx 0.9804,\text{ for all } m\geq 211.$$ Now for $m\geq 3$, $\phi(m)\leq m-1$ and hence $\frac{\sqrt{m}}{\sqrt{2}}\leq \phi(m)\leq m-1$ for $m\geq 3$.   Hence, from Table~\ref{tab:fermat-reductions}, for $m\geq 8$, one has
$$1-\frac{\sqrt{2}}{\sqrt{m}}\leq \delta^{\rm non-hw}(F_{n,m})=1-\frac{1}{\phi(m)}\leq 1-\frac{1}{m-1},$$ which  gives the asserted limit as $m\to\infty$.
\ep

\br So \Cref{th:frequency-of-non-hodge-witt-cm} shows that contrary to the popular understanding of primes of ordinary reduction, the de Rham-Witt cohomology of the Fermat hypersurface $F_{n,m}$ has  infinite torsion (in de Rham-Witt cohomology) with a far higher frequency than the frequency of primes of ordinary or Hodge-Witt reductions. In particular, the reductions are quite far removed from being ordinary or Hodge-Witt most of the time and  \emph{in the CM situation infinite torsion in de Rham-Witt cohomology  is the norm rather than a rare appearance.}

So one could say that if $X/K$ is a smooth projective variety such the Mumford-Tate group of $H^n_{dR}(X/K)$ has a factor which is a torus then one should expect that $X$ has a larger fraction (even of positive density) of primes of non-Hodge-Witt reduction and a randomly picked prime has a good chance of being a prime of non-Hodge-Witt reduction!
\er

\subsection{Surfaces: abelian or general type}\label{ss:abvar-gen-typ-surf}
Here are some examples of curves and surfaces.
\begin{example}\label{elliptic-curve-example}[Product of two elliptic curves]
Let $E/\Q: y^2=x^3-x$ and $E'/\Q: y^2=x^3+1$ be elliptic curves over $\Q$ with complex multiplication by $\Q(i)$ and $\Q(\zeta_3)$ respectively. Then $E$ has ordinary reduction if and only if $p\cong1\bmod{4}$ and $E'$ has ordinary reduction if and only if $p\cong1\bmod{3}$. Let $A=E\times E'$. This is an abelian surface, and we want to calculate the densities $\dora{A}$ and $\dhwa{A}$ or primes of good ordinary reduction and Hodge-Witt reduction for $A$. Now suppose $p$ is a prime of good reduction for $A$. Then we see from the proof of \Cref{th:equivalence} that $A$ has ordinary reduction at $p$ if and only if both $E$ and $E'$ have good ordinary reductions at $p$ (so $p\cong1\bmod{3}$ and $p\cong1\bmod{4}$). Thus, we see that $\dora{A}=\frac{1}{4}$. Again from the proof of \Cref{th:equivalence}, we see that $A$ has Hodge-Witt reduction if and only if one of $E,E'$ is ordinary and the other is Hodge-Witt. Now by \Cref{pr:curves-are-hodge-witt},  both $E,E'$ have Hodge-Witt reduction at any prime of good reduction. Thus $A$ has Hodge-Witt reduction at $p$ whenever $E$ has good ordinary reduction at $p$ or whenever $E'$ has good ordinary reduction at $p$. Thus, $A$ has Hodge-Witt reduction at $p$ if and only if $p\cong1\bmod{4}$ or $p\cong1\bmod{3}$. Hence, we see that $$\dhwa{A}=\frac{1}{2}+\frac{1}{2}-\frac{1}{4}=\frac{3}{4}.$$ Thus, we see that for Abelian surfaces it is possible to have $$\frac{1}{4}=\dora{A}<\dhwa{A}=\frac{3}{4}<1.$$
In particular, the set of primes of Hodge-Witt reduction includes the set of primes of ordinary reduction, but can be strictly larger than it. 
\end{example}

\begin{example}\label{elliptic-curve-example2}[Product of two elliptic curves II]
Let $E/\Q: y^2=x^3-x$ and $E'=E$ and $A=E\times_\Q E'=E\times_\Q E$. Here $E$ has CM by $\Q(i)$. Let $p$ be a prime of good reduction for $A$.	Then $A$ has ordinary reduction at $p$ if and only $A$ has Hodge-Witt reduction at $p$ if and only if $E$ has ordinary reduction at $p$ if and only if $p\cong 1\bmod{4}$. Thus, $\dora{A}=\frac{1}{2}=\dhwa{A}<1$.
\end{example}

\brem 
The reader will have no difficulty in generalizing the above two examples to arbitrary number of elliptic curve factors of various type.
\erem

\begin{example}\label{ex:hodge-witt-cm-galois-example}[CM by a field galois (over $\Q$)]
Let $C:y^2=x^5+1$. Here the genus $g=2$  and the Jacobian of $C$ has  CM by $\Q(\zeta_5)$ which is cyclic and galois over $\Q$. Assume $p>5$. Then  by \cite[3.6.1]{goren10} its Jacobian is ordinary if and only if $p\cong{1}\bmod{5}$, and for all other residue classes modulo $5$, the reduction is non-Hodge-Witt. In particular, in this example, the set of primes of ordinary reduction coincides with the set of primes of Hodge-Witt reduction  and this common set of primes has density $\dora{A}=\frac{1}{\phi(5)}$ (where $\phi(5)=|\gal(\Q(\zeta_5)/\Q)|=4$), and the set of primes of non-Hodge-Witt reduction has density $\dnhwa{A}=1-\dora{A}=1-\dhwa{A}=1-\frac{1}{\phi(5)}=\frac{4}{5}$.
\end{example}

\begin{example}\label{ex:hodge-witt-cm-example}[CM by a non galois CM field]
Consider the field $K=\Q[x]/(x^4+134x^2+89)$. This is a quartic non-galois CM field, its normal closure $N$ has $G=\gal(N/\Q)=D_4$ the dihedral group of order $8$. This field and abelian surfaces with CM by this field are discussed in \cite[3.5]{goren10}. Examining the conjugacy classes of $G$ we see that the conjugacy class of transpositions has two elements, and the conjugacy class of the identity is of course one element. Now there exists an abelian variety $A$ over some number field $L$, with CM by $K$.  The list of possible reductions at primes of $L$ is described in \cite[Table 3.5.1]{goren10} and is determined by the factorization of $p$ in $K$. In that table, the primes with ordinary reduction correspond to entries with $f=2$, while primes with Hodge-Witt reduction correspond to $f\geq 1$ ($f$ is the $p$-rank of the reduction modulo a prime lying over $p$). The table row numbers $(i), (v)$ are the only entries in \cite[Table 3.5.1]{goren10} with ordinary reduction. The row numbers (i), (ii), (iv), (v), (xii) and (xiv) of \cite[Table 3.5.1]{goren10} are all the possible entries with Hodge-Wit reduction. Note in the entries (xii), (xiv), $p$ ramifies in $K$ (and hence in $N$) and so this set of primes is finite and can be ignored. For $A$, by Theorem~\ref{thm:cm-hodge-witt-2} the set of primes of Hodge-Witt reduction has density  $\frac{3}{4}\geq\delta^{hw}\geq \frac{3}{8}$ while the density of primes of ordinary reduction is $\frac{1}{2}\geq\delta^{ord}\geq\frac{1}{8}$. We note that some primes  other than the almost split primes contribute to the densities,  and one can, using \cite{goren10}, describe them explicitly, but we do not work this out here--as this sort of calculation should be worked out in some generality and this will require some additional combinatorics which belongs to a separate paper by itself. For all other sufficiently large primes, the reduction is non-Hodge-Witt and there are infinitely many of these as well (from the above bounds).
\end{example}

\begin{example}[Surfaces of general type]\label{surf-gen-type}
Suppose $C,C'$ are smooth, projective curves over $\Q$. Assume that $C$ is curve of genus $g_C\geq2$, and $C'$ has genus $g_{C'}=2$. Assume that the Jacobian of $C$ is a simple abelian variety with Complex multiplication.  Assume that the Jacobian of $C$ (resp. $C'$) is not isogenous to a factor of the Jacobian of $C'$ (resp. $C$). Since the Jacobian of $C$ has CM,  $C$ has good ordinary reduction at a positive density of rational primes. By \cite[Corollary 2.9]{ogus82}, $C'$ has good ordinary reduction at a positive density of  primes. Let $\dora{C}, \dora{C'}$ be the density of primes of good ordinary reduction for $C,C'$ respectively. Thus  $\dora{C},\dora{C'}>0$. Now consider the smooth, projective surface $S=C\times_\Q C'$. By an argument similar to the one given in \Cref{elliptic-curve-example}, \Cref{elliptic-curve-example2} and \Cref{pr:curves-are-hodge-witt},  $S$ is a surface with good Hodge-Witt reduction whenever $C$ or $C'$ has ordinary reduction. Moreover, $S$ has good ordinary reduction at $p$ precisely when both $C,C'$ have good ordinary reduction modulo $p$.  Let $\dora{S}$ (resp. $\dhwa{S}$) be the density of primes of good ordinary reduction for $S$ (resp. Hodge-reduction for $S$), let $\dora{C,C'}$ be the set of primes where both $C,C'$ have good, ordinary reduction. Then 
$$\dora{S}=\dora{C,C'},$$
and
$$\dhwa{S}=min(1,\dora{C}+\dora{C'}-\dora{C,C'})\geq0,$$
(the non-negativity follows from the tautological inequalities $\dora{C}\geq \dora{C,C'}$ and $\dora{C'}\geq \dora{C,C'}$)
and in general $$\dhwa{S}\geq \dora{S}$$
with strict inequality possible by choosing $C'$ to have CM Jacobian by a CM field not isomorphic to the CM field corresponding to $C$.

For specific examples, consider $C:y^2=x^5+1$ (CM by $K=\Q(\zeta_5)$ and LMFDB Label 10000.b.800000.1) and $y^2 + x^3y = -2x^4 - 2x^3 + 2x^2 + 3x - 2$ (CM by $K'=\Q(\alpha)$ where $\alpha$ is a root of $ x^{4} - x^{3} + 2x^{2} + 4x + 3 =0$ and LMFDB Label 28561.a.371293.1).  These curves are not $\bar{\Q}$-isomorphic as can be seen from their non-isomorphic CM fields or from their Igusa invariants listed in the data for these curve provided by \cite{lmfdb}. Now $\dora{C,C'}$ corresponds precisely to the set of primes $p$ which are completely split in both the CM fields $K$ and $K'$. As both $K,K'$ are cyclic Galois extensions of $\Q$ (each of degree four), these fields are uniquely identified by their respective sets of primes which split completely \cite[Corollary 13.10]{neukirch-algebraic-book}. Let $K^+\subset K$ (resp. $K^{'+}\subset K'$) be the totally real subfield of $K$ (resp. $K'$). Then we have $[K:\Q]=4=2\cdot 2=[K:K'][K':\Q]$ (and a similar assertion for $K'$). From \cite[Table 3.3.1]{goren10}, one sees that a prime $p$ which splits completely in $K$ (resp. $K'$) splits completely in $K^+$ (resp. $K^{'+}$).   Because one is dealing with successive quadratic extensions $K\supset K^+\supset \Q$ (and similarly for $K'$), we can calculate the density of all the required sets of primes.  From this we  see that $\dora{C,C'}\neq \dora{C}$ or $\dora{C'}$ as $\dora{C,C'}$ is the density of the set of primes $p$ which are completely split in both the CM fields $K$ and $K'$.  Hence $\dhwa{S}>\dora{S}$. Thus, we have examples of surfaces of general type where \Cref{con:positive-density-hodge-witt} and Conjecture~\ref{serre} hold and provide different sets of primes.
\end{example}

\subsection{Abelian Threefolds}
\begin{example}\label{ex:ab-thrfld}
Let $E_1,E_2,E_3$ be the three  elliptic curves over $\Q$ with CM by pairwise non-isomorphic imaginary quadratic fields $K_1,K_2,K_3$ (which in particular implies that these elliptic curves are pairwise non-isomorphic). Consider the abelian threefold $X=E_1\times E_2\times E_3$ over $\Q$. Then \Cref{th:positive-density-hodge-witt-abelian-threefold} says that there exist infinitely many primes $p$ at which $X$ has Hodge-Witt reduction. By \Cref{pr:curves-are-hodge-witt} each of $E_j$ has Hodge-Witt reduction for all but finitely many primes $p$ and hence the method of proof of \Cref{th:equivalence} (i.e. by \cite[III, Prop 2.1(ii) and Prop 7.2(ii)]{ekedahl85}) we see that $X$ has Hodge-Witt reduction if and only if any pair of $E_1,E_2,E_3$ has ordinary reduction at $p$ i.e. if and only if $p$ splits completely in any pair of $K_1,K_2,K_3$. On the other hand, $X$ has ordinary reduction at $p$ if and only each of $E_1,E_2,E_3$ has ordinary reduction at $p$ i.e. if and only if $p$ splits completely in each of the fields $K_1,K_2,K_3$. From this the densities $\dhwa{X}$ and $\dora{X}$ can be easily calculated using the inclusion exclusion principle. This gives $\dhwa{X}=\frac{5}{8}$ and $\dora{X}=\frac{1}{8}$.
\end{example}

\subsection{Abelian fourfolds}
I thank the Referee for suggesting \Cref{ex:ab-fourfold}. I have also added \Cref{ex:ab-fourfold2} to provide interesting contrast.
\newcommand{\leg}[1]{\left(\frac{#1}{p}\right)}
\newcommand{\notdiv}{\hspace{-4pt}\hspace{2pt}}
\begin{example}\label{ex:ab-fourfold}
It may happen that $X$ is an abelian variety with CM, but $\delta^{ord}_X=\delta^{hw}_X$ i.e. every prime of Hodge-Witt reduction is a prime of ordinary reduction equivalently  $X$ has no primes of almost ordinary reduction. To construct such an $X$, let $d_1,d_2,d_3,d_4$ mutually distinct square-free, negative integers (i.e.  $d_i<0$ for $i=1,2,3,4$) and $d_4=d_1\cdot d_2\cdot d_3$. Thus, $K_i=\Q(\sqrt{d_i})$ are four imaginary quadratic fields. Let $E_i$ be an elliptic curve with CM by $K_i$ and let $X=E_1\times E_2\times E_3\times E_4$. Let $\leg{d_i}$ be the Legendre symbol of $d_i$ with respect to $p$. Then $X$ has complex multiplication and if $p\nmid d_4$ then we must have $\leg{d_1}\leg{d_2}\leg{d_3}\leg{d_4}=1$. This means three of the Legendre symbols cannot be equal to $-1$. If $p|d_4$ then $p$ divides one of the remaining $d_i$, say $p|d_1$. Then $E_1,E_4$ have supersingular reduction over $p$ and hence $X$ cannot have almost ordinary at primes over $p$. Hence, up to a finitely many possible exceptions, primes of ordinary and Hodge-Witt reduction coincide and so $\delta^{\rm ord}_X=\delta^{\rm hw}_X$ and for all but finitely many primes $p$, either $p$ is a prime of ordinary reduction or $X$ has $p$-rank $\leq 2=\dim(X)-2$ i.e. $p$ is a prime of non-Hodge-Witt reduction for $X$.
\end{example}

\begin{example}\label{ex:ab-fourfold2}
 Let $d_1,d_2,d_3,d_4$ be mutually coprime,  square-free, negative integers (i.e.  $d_i<0$ for $i=1,2,3,4$  and $(d_i,d_j)=1$  for all $i\neq j$). Thus, $K_i=\Q(\sqrt{d_i})$ are four imaginary quadratic fields, and  $K_i\cap K_j=\Q$ for all $i\neq j$. For simplicity, we assume that all these fields have class number one. For $1\leq i\leq 4$, let $E_i$ be an elliptic curve over $\Q$ with CM by $K_i$ and let $X=E_1\times E_2\times E_3\times E_4$. 	Then $X$ has good, ordinary reduction modulo $p$ if and only if $p$ splits completely  in all the fields $K_i$ (i.e. $\leg{d_i}=1$) for $1\leq i\leq 4$. Since each elliptic curve $E_i$ has Hodge-Witt reduction at all but finitely many primes (\Cref{pr:curves-are-hodge-witt}), and so $X$ has Hodge-Witt reduction if and only if any three of the $E_1,E_2,E_3,E_4$ have good ordinary reduction. Thus, in contrast to \Cref{ex:ab-fourfold}, one sees that $$\delta^{\rm ord}_X=\frac{1}{2^4}=\frac{1}{16} < \delta^{\rm hw}_X=\frac{5}{16}<1.$$ 
 
 Let me make this completely explicit. Let $d_1= -7,d_2=-11,d_3=-19,d_4=-43$ and let $K_i=\Q(\sqrt{d_i})$ for $i=1,2,3,4$. Then each $K_i$ is a CM field of class number one and  $E_i$ has CM by $K_i$ and is defined over $\Q$, and here $P=K_1\times K_2\times K_3\times K_4$. Simple computation for primes in the range $2000\leq p\leq 2500$, which has $64$ primes in it, reveals the following. For each $p\in \{ 2351, 2381, 2383\}$, one has $\leg{d_i}=1$ for all $i$ and hence $X$ has ordinary reduction at these primes $p$, and this is the complete list of primes at which $X$ has ordinary reduction with $2000\leq p\leq 2500$. On the other hand,  for each $p\in \{2069, 2131, 2137, 2297, 2371\}$, one checks easily that $$\leg{d_i}=1 \text{ for } i=1,2,3 \text{ and }\leg{d_4}=-1.$$ Hence,  $X$ has Hodge-Witt but non-ordinary reduction for these primes $p$.  The complete list of primes of Hodge-Witt reduction in this prime range is 
 \begin{multline*}
 \{2063, 2069, 2113, 2131, 2137, 2237, 2239, 2243, 2287, 2293, \\ 2297, 
 2333, 2351, 2357, 2371, 2381, 2383, 2411, 2417, 2459\},
 \end{multline*}
 and the remaining $41$ primes in this range are primes of non Hodge-Witt reduction. \emph{However}, all the field $K_i$ in this example are quadratic fields and \Cref{def:almost-completely} treats quadratic fields differently from CM fields of higher degree and hence in this example $p$ splits almost completely in $P$ if and only if $p$ splits completely in $P$. The next example remedies this defect.
\end{example}

\begin{example}\label{ex:ab-fourfold3}
Consider the CM algebra $P=K_1\times K_2\times K_3$ where $K_1=\Q(\sqrt{-7}), K_2=\Q(\sqrt{-11})$ and $K_3=\Q[x]/(x^4+134x^2+89)$ be the CM field used in \Cref{ex:hodge-witt-cm-example}. Then $\dim(P)=8$ and let $X=E_1\times E_2\times A$ be the abelian fourfold where $E_1$ and $E_2$ are CM elliptic curves over $\Q$ with CM by $K_1$ and $K_2$ respectively and $A$ is the  abelian surface with CM by $K_3$ used in \Cref{ex:hodge-witt-cm-example}. Thus $\dim(X)=4$ and $X$ and one can take $X$ to be defined over a number field $L$ containing the field of  definition of all the factors and also containing all the $K_i$; moreover $X$ has CM by $P=K_1\times K_2\times K_3$. Consider primes $p$ in the range $5000\leq p\leq 7000$. There are $231$ primes in this range. The rational primes $p$ which split completely in $P$ are the nine primes $$p\in \{5273, 5413, 6073, 6301, 6329, 6581, 6857, 6869, 6997\}.$$ The primes which split almost completely but not completely in $P$ are $$p\in\{ 5021, 5237, 5569, 5581, 5861, 6029, 6197, 6449, 6469, 6521, 6637, 6689, 6829\},$$ and the primes $p$ which split almost completely in $P$ are $22$ primes:
\begin{multline*}
	p\in \{5021, 5237, 5273, 5413, 5569, 5581, 5861, 6029, 6073, 6197, \\ 6301, 6329, 6449, 6469, 6521, 6581, 6637, 6689, 6829, 6857,\\ 6869, 6997\}.
\end{multline*}
 The remaining $209$ of the $231$ primes do not split almost completely in $P$. Suppose $\wp|p$ and $X$ has good reduction at $\wp$, then the  reduction types: ordinary, Hodge-Witt but non-ordinary, and Hodge-Witt, can be read off from these three lists; otherwise $X$ has non Hodge-Witt reduction at $\wp$ (for $p$ in this range).
\end{example}

\subsection{A variety of general type with $\delta^{\rm ord}_X<\delta^{\rm hw}_X<\varepsilon$ and $\delta^{\rm hw}_X/\delta^{\rm ord}_X>c>1$ }
\begin{example}\label{ex:gen-type-negligible-den}
Let $11=p_1<p_2=13<p_3=17<\cdots<p_m$ be the first $m$ primes greater than $7$. Let  $K_j=\Q(\zeta_{p_j})$ for $j=1,\ldots,m$. Then the fields $K_j$ are pairwise linearly disjoint over $\Q$. Let $C_j$ be the Fermat curve  $C_j=F_{1,p_j}\subset \P^2$ of degree $p_j\geq 11$ for $j=1,\ldots,m$. Let 
\be\label{ex:gen-type-exmpl} X=\prod_{j=1}^m C_m.\ee
By \Cref{tab:fermat-reductions}, $C_j$ has good ordinary reduction if and only if  $q\cong 1\bmod(p_j)$ and hence $\delta^{ord}(C_j)=\frac{1}{p_j-1}$. 
Now by \cite[III, Prop 2.1(ii) and Prop 7.2(ii)]{ekedahl85}, $X$ has ordinary reduction modulo a prime $q>p_m$ if and only if all the factors  $C_j$ have good ordinary reduction modulo $q$.  Thus, by the Chinese Remainder Theorem, $q\cong 1\bmod{ (\prod_{j=1}^m p_j)}$ and so one has 
\be\label{eq:den1} \delta^{\rm ord}_X=\prod_{j=1}^m\frac{1}{p_j-1}.\ee
On the other hand, by \Cref{pr:curves-are-hodge-witt} and \cite[III, Prop 2.1(ii) and Prop 7.2(ii)]{ekedahl85}, $X$ has good Hodge-Witt reduction whenever any $m-1$ of the $C_j$ have good ordinary reduction. The set of such primes is a disjoint union of the set of primes $q\cong  1\bmod{ (\prod_{j=1}^m p_j)}$ (this is the set of primes of good ordinary reduction for $X$ and its density is given by \eqref{eq:den1}) and the sets of primes $q\cong  1\bmod{ (\prod_{j=1, j\neq i}^m p_j)}$, but $ q\not\cong 1\bmod{p_i}$ for each $i=1,2\ldots,m$ (each of these $m$ sets consists of primes of Hodge-Witt but non-ordinary reduction). The density of each of the latter sets is easily computed using their description. Thus, one has

\be\label{eq:den2}
\delta^{\rm hw}_X = \delta^\text{ord}_X + \sum_{j=1}^{m} \frac{p_j - 2}{\prod_{i=1}^{m} (p_i - 1)} 
= \frac{1 + \sum_{j=1}^{m} (p_j - 2)}{\prod_{i=1}^{m} (p_i - 1)}
= \frac{\sum_{j=1}^{m} p_j - 2m + 1}{\prod_{i=1}^{m} (p_i - 1)}.
\ee
Moreover, $\delta^{\rm non-hw}_X=1-\delta^{\rm hw}_X$.  For $m$ sufficiently large, the first density \eqref{eq:den1} can be estimated using the prime number theorem to be about $e^{-p_m}\approx m^{-m}$, and the second density \eqref{eq:den2} can be estimated using to be about $\frac{m^2\log(m)}{2}\cdot e^{-p_m}$.  Comparing the discrepancy in the growth of these two as $m\to\infty$, allows one to arrive at the following assertion: 

Given a real numbers $0<\varepsilon<1$, and $c>1$, there exists an integer $m_0=m_0(\varepsilon,c)\gg 1$ and a geometrically connected, smooth, projective variety $X=X(\varepsilon,c)$ defined over $\Q$ and given by \eqref{ex:gen-type-exmpl}, such that $X$ is of general type with $\dim(X)=m\geq m_0$, and one has  
$$0<\delta^{\rm ord}_X<\delta^{\rm hw}_X<\varepsilon,$$
$$\delta^{\rm non-hw}_X=1-\delta^{\rm hw}_X>1-\varepsilon,$$
and
$$\delta^{\rm hw}_X/\delta^{\rm ord}_X>c>1.$$

For example, for $\varepsilon=10^{-4}$ and $c=10$, one can take $m=5$.
\end{example}

\bibliographystyle{plainnat}

\end{document}